%% file: ChernModulusNMJv3_numbering.tex
\documentclass{amsart}

\usepackage{amsmath,amssymb,mathrsfs}
\usepackage{comment}
\usepackage{amsthm}
\usepackage[all]{xy}
\usepackage{array}
\usepackage[shortlabels]{enumitem}
\usepackage[dvipdfmx]{hyperref}
\usepackage{amscd}
\usepackage{graphicx}
\theoremstyle{plain}
\newtheorem{theorem}{Theorem}[section]
\newtheorem*{theorem*}{Theorem}
\newtheorem{corollary}[theorem]{Corollary}
\newtheorem*{corollary*}{Corollary}
\newtheorem{proposition}[theorem]{Proposition}
\newtheorem*{prop*}{Proposition}
\newtheorem{lemma}[theorem]{Lemma}
\newtheorem*{lemma*}{Lemma}
\newtheorem{claim}[theorem]{Claim}
\newtheorem*{claim*}{Claim}

\newtheorem*{conclusion*}{Conclusion}

\newtheorem{def-lem}[theorem]{Definition-Lemma}

\theoremstyle{definition}
\newtheorem{definition}[theorem]{Definition}
\newtheorem*{definition*}{Definition}

\newtheorem*{cnst*}{Construction}

\newtheorem*{pers*}{Perspective}

\newtheorem*{dig*}{Digression}

\newtheorem*{question*}{Question}
\newtheorem*{notation*}{Notation}
\newtheorem{notation}[theorem]{Notation}

\newtheorem*{expectation*}{Expectation}

\theoremstyle{remark}
\newtheorem{remark}[theorem]{Remark}
\newtheorem*{remark*}{Remark}

\newtheorem*{example*}{Example}

\newcommand{\bb}[1]{\mathbb{#1}}
\newcommand{\mcal}[1]{\mathcal{#1}}
\newcommand{\bbZ}{\bb{Z}}

\newcommand{\bbP}{\bb{P}}

\newcommand{\super}[1]{^{(#1)}}
\newcommand{\mrm}[1]{\mathrm{#1}}
\newcommand{\mbf}[1]{\mathbf{#1}}
\newcommand{\Div}{\mrm{div}}

\newcommand{\seqdots}[3]{{#3_{#1},\dots ,#3_{#2}}}
\newcommand{\Spec}{\mrm{Spec}}

\newcommand{\supers}[1]{^{(#1)*}}
\newcommand{\supero}[1]{^{(#1)\circ}}

\newcommand{\superos}[1]{^{(#1)\circ *}}
\newcommand{\os}{^{\circ *}}
\newcommand{\GL}{\mrm{GL}}
\newcommand{\mfr}[1]{\mathfrak{#1}}
\newcommand{\msf}[1]{\mathsf{#1}}

\newcommand{\rel}{{\mathsf{rel}}}

\newcommand{\Sm}{{\mathsf{Sm}}}
\newcommand{\MSm}{{\mathsf{MSm}}}

\newcommand{\bhom}{{\mathbf{hom}}}
\newcommand{\CH}{{\mrm{CH}}}

\newcommand{\cone}{{\mathrm{cone}}}
\newcommand{\Nis}{{\mrm{Nis}}}

\newcommand{\forget}{{}}
\newcommand{\Volo}{{\mathbf{X}}}
\newcommand{\cupp}{\cdot } 

\newcommand{\Psp}{\bbP ^{r-1}}
\newcommand{\Oone}{\mcal{O}(1)}
\newcommand{\ps}{{p_*}}
\newcommand{\vertex}[2]{{v^{[#2 ]}_{#1} }}

\newcommand{\PSh}{\mathsf{PSh}}

\newcommand{\Set}{\mathsf{Set}}

\newcommand{\Ho}{\mathsf{Ho}}
\newcommand{\ch}{\mathsf{ch}}
\newcommand{\id}{\mathrm{id}}
\newcommand{\pro}{\mathrm{pro}\mathchar`-}
\newcommand{\Dtimes}{\overset{\mrm{:D}}{\otimes}}

\DeclareMathOperator{\Hom}{Hom}

\DeclareMathOperator*{\holim}{holim}
\DeclareMathOperator*{\hocolim}{hocolim}
\DeclareMathOperator{\hofib}{hofib}

\DeclareMathOperator{\HC}{HC}
\DeclareMathOperator{\HN}{HN}

\begin{document}

\title{Chern classes with modulus}



\author{Ryomei Iwasa}	\address{Department of Mathematical Sciences, University of Copenhagen, Universitetsparken 5, DK-2100 Copenhagen \O.}	\email{ryomei@math.ku.dk}
\thanks{}
\author{Wataru Kai}	\address{Mathematical Institute, Tohoku University, Aza-Aoba 6-3, Sendai 980-8578, Japan.}	\email{kaiw@tohoku.ac.jp}
\thanks{}


\maketitle

\begin{abstract}
In this paper, we construct Chern classes from the relative $K$-theory of modulus pairs to the relative motivic cohomology defined by Binda-Saito.
An application to relative motivic cohomology of henselian dvr is given.
\end{abstract}

\tableofcontents

\section*{Introduction}

Algebraic cycles with modulus have been considered to broaden Bloch's theory of algebraic cycles \cite{Bl86}.
This concept has arisen from the work by Bloch-Esnault \cite{BE}, and now it is fully generalized by Binda-Saito in \cite{BS}.
The purpose of this paper is to relate Binda-Saito's theory of algebraic cycles to algebraic $K$-theory by establishing a theory of Chern classes.
To be precise, we prove the following.
\begin{theorem}{\upshape (Theorem \ref{thm:ChernModulus}, Theorem \ref{thm:WhitneyK})}
\label{theorem0}
Let $X$ be an equidimensional scheme of finite type over a field $k$ and $D$ an effective Cartier divisor on $X$ such that $X\setminus D$ is smooth over $k$.
Then, for $i,n\ge 0$, there exist maps
\[
	\msf{C}_{n,i} \colon \mrm{K}_n(X,D) \to H^{2i-n}_{\mcal{M},\mrm{Nis}}(X|D,\bb{Z}(i))
\]
from the relative algebraic $K$-theory to the (Nisnevich) relative motivic cohomology as defined in \cite[Definition 2.10]{BS}.
These maps are functorial in $(X,D)$ in the category of modulus pairs $\MSm$ (Definition \ref{DefMSm}) and coincide with Bloch's Chern classes \cite[\S7]{Bl86} when $D=\emptyset$.
Furthermore, $\msf{C}_{n,i}$ are group homomorphisms for $n>0$ and satisfy the Whitney sum formula for $n= 0$.
\end{theorem}


Comparison maps between certain parts of relative algebraic $K$-theory and (additive) higher Chow groups with modulus had been constructed in some cases by authors such as Bloch-Esnault, R\"ulling, Park, Krishna-Levine, Krishna-Park, Krishna, R\"ulling-Saito and Binda-Krishna \cite{BE, Ru, Pa, KL, KP15, Kr, RS, BiKr} (to name a few), who reveal profound aspects of those maps.
But this is the first time a comparison map has been given on the entire (non-negative) range and in such generality.

As an application, we give a partial result for the comparison of relative algebraic $K$-theory and relative motivic cohomology for henselian dvr.

\begin{theorem}{\upshape (Theorem \ref{th:henseliandvr})}
\label{theorem1}
Let $X$ be the spectrum of a henselian dvr over a field of characteristic zero and $D$ the closed point of $X$ seen as a Cartier divisor.
Then, for every $n\ge 0$, there is a natural isomorphism
\begin{multline*}
	\{\CH^*(X|mD,n)\}_{m,\bb{Q}} \\
		\simeq \{\mrm{K}_n(X,mD)\oplus \ker(\CH^*(X|mD,n) \to \CH^*(X,n)) \}_{m,\bb{Q}}
\end{multline*}
in the category $(\pro\msf{Ab})_{\bb{Q}}$ of pro abelian groups up to isogeny.
\end{theorem}

\subsection*{Acknowledgements}

When the authors were graduate students, Shuji Saito, Kanetomo Sato and Kei Hagihara encouraged the authors to learn $K$-theory techniques like the Volodin space and stability results, and moving techniques of algebraic cycles, which later turned out indispensable in carrying out this project.
Part of the work was done while one or both of the authors were staying at the Universit{\"a}t Regensburg on several occasions. 
Conversations with Federico Binda and Hiroyasu Miyazaki were helpful.
We owe a lot to the referee for improvement in exposition.
The work was supported by JSPS KAKENHI Grant Number 15J02264, 16J08843, and by the Program for Leading Graduate Schools, MEXT, Japan.

\input{ProjFormulaNMJv2}

\input{UnivChernNMJv2}

\input{WhitneyNMJv2}

\input{ApplicationsNMJv2}

\appendix
\setcounter{equation}{0}
\def\theequation{\Alph{section}\arabic{equation}}

\input{LocHtpNMJv2}

\setcounter{equation}{0}
\input{CyclePrelimNMJv2}

\end{document}

%% file: ProjFormulaNMJv2.tex
\section{Global projective bundle formula}\label{GlobalProjBund}

The aim of this section is to formulate and prove a projective bundle formula for the cycle complex with modulus as formulated in 
Theorem \ref{ProjBund} below. 
It takes place in a very global set-up. 
As such, it requires a considerable amount of effort to get all the compatibility right just to define the map. Once defined, the proof that it is an isomorphism is then a local problem and already essentially known.


\subsection{Modulus pairs and cycle complex presheaves}\label{SetUpProjBund}
We begin by the definition of the categories of modulus pairs.
\begin{definition}\label{DefMSm}
Let $k$ be an arbitrary base field.
Denote by $\MSm $ the category of pairs $(X,D)$ of an equidimensional $k$-scheme of finite type and an effective Cartier divisor on it,
such that $X^\circ := X\setminus D$ is smooth.
(Such pairs are commonly called ``modulus pairs''.)
Morphisms $f\colon (X',D')\to (X,D)$ are the morphisms of $k$-schemes $X'\to X$ which restrict to morphisms $D'\to D$ of subschemes.

We give it a (pre)topology, which we call the {\em Nisnevich topology}, by declaring that a family of morphisms
$\{ f_i\colon (X_i,D_i)\to (X,D) \} _i$ is a covering if and only if
the underlying family $\{ f_i\colon X_i\to X \} _i$ is a Nisnevich cover and $D_i=f_i^*D$ holds for all $i$.

\end{definition}

\begin{remark}
This is not the same as the category denoted by the same symbol in \cite{KSY}.
First, we ask the scheme-morphism $X'\to X$ to be defined on the entire $X'$ rather than the open part $X^{\prime\circ }$.
Second, the condition on divisors are also different:
For example, the identity morphism on $X$ induces a morphism
$(X,\emptyset )\to (X,D)$ in our category and $(X,D)\to (X,\emptyset )$ in theirs. 
We nonetheless opted to use the concise symbol $\MSm $.
\end{remark}

We would like to have a variant $\MSm ^* $ of it on which cycle complexes with modulus are functorial.
Our specific choice below is not crucial in this work.
As another possible choice of $\MSm ^*$, one could probably take a version of Levine's $\mcal{L}(\mrm{Sm} )$ \cite[p.9]{Levine}.

\begin{definition}
Fix a category $\Lambda $ with only finitely many objects and morphisms, and a functor $F\colon \Lambda \to \MSm $; $ \lambda \mapsto (X_\lambda ,D_\lambda )$.
Let $\MSm ^*:= \MSm /F$ be the site fibered over $F$ (Definition \ref{def:site_fibered} via Yoneda), with the additional condition that the underlying morphism $f\colon X\to X_\lambda $ is \'etale and that $D=f^*D_\lambda $.
\end{definition}

Since the dependence on $F$ does not play a major role in this article, we do not make it explicit in the notation.
Note that we have an obvious forgetful functor $\MSm ^*\to \MSm $ given by
$((X,D),\lambda ,f)\mapsto (X,D)$.

\begin{remark}
The principal case to have in mind is when $\Lambda $ is just a point. Let $(X,D)$ be the value of this point. 
Then our $\MSm ^*$ is nothing but the small Nisnevich site $(X,D)_\Nis $. This case is enough for constructing the Chern classes for each pair $(X,D)$.
To get the functoriality of the Chern classes as in Theorem \ref{theorem0}, we need to consider the category $\Lambda = \{ *\longrightarrow * \} $ with a unique non-identity morphism.
Larger $\Lambda $'s may be useful when one considers more involved compatibility.

\end{remark}

We refer the reader to \cite{BS}
for the definition of Binda-Saito's cycle complex with modulus $z^i(X|D,\bullet )$.

\begin{definition}
Let $i\ge 0$ be a non-negative integer.
For each $((X,D),\lambda ,f)\in \MSm ^*$, denote by 
\begin{equation*}
z^i_\rel ((X,D),\lambda ,f;\bullet )\subset z^i(X|D,\bullet )
 \end{equation*}
the subcomplex of cycles
$V\in z^i(X|D,n)$ such that for every morphism
\begin{equation*}
(g,\varphi )\colon ((X',D'),\lambda ',f')\to ((X,D),\lambda ,f) \quad \text{ in } \MSm ^*, \end{equation*}
its pull-back $g^*V\in z^i(X'|D',n)$ is well-defined.
This defines a presheaf $z^i_\rel $ of complexes on $\MSm ^*$, which we call the codimension $i$ {\em cycle complex presheaf} on $\MSm ^*$.
We remark that by a moving lemma with modulus (Theorem \ref{MovingVariant}),
the inclusion
$ z^i_\rel ((X,D),\lambda ,f ; \bullet )$ 
$\hookrightarrow z^i(X|D,\bullet )$
is a quasi-isomorphism locally in the Nisnevich topology on each $X$.
\end{definition}

Next, we want to define a presheaf of complexes $p_*z^r_\rel $ which serves as ``the cycle complex of the projective bundle associated to the universal vector bundle on $B\GL _r$''.
Defining it requires some more notation which we now introduce.

For a non-negative integer $n\ge 0$, set $[n]= \{ 0,\dots ,n \} $ and endow it with the usual order.
For a non-negative integer $r\ge 0$, let us recall (or {\em adopt the convention}) that $B\GL _r$ is the simplicial $k$-scheme $B_n\GL _r := (\GL _r)^n $ with the structure morphism associated to each ordered map $\theta \colon [m]\to [n]$:
\begin{equation*}
(\GL _r)^n \to (\GL _r)^m ; \quad
(\alpha _1,\dots ,\alpha _n )\mapsto (\alpha _{\theta (j-1)+1} \cdots  \alpha _{\theta (j)}) _{1\le j\le m} .
 \end{equation*}
The simplicial scheme $B\GL _r$ defines a simplicial presheaf on $\MSm ^*$ by
$((X,D),\lambda ,f)\mapsto B\GL _r(X)$.
%

\begin{definition}
Let $\MSm ^*/B\GL _r$ be the site fibered over 
$B\GL _r$. 
\end{definition}

Recall that the simplicial $k$-scheme $\bbP (E\GL _r)$ has $\bbP (E_n\GL _r)=\bbP ^{r-1}\times (\GL _r )^n$ as its $n$-th component, and the structure morphism corresponding to $\theta \colon [m]\to [n]$ is defined by 
\begin{equation*}
(z,\alpha _1,\dots ,\alpha _n)\mapsto (z\alpha _1\dots \alpha _{\theta (0)} ~,~\alpha _{\theta (0)+1}\dots \alpha _{\theta (1)}~,~\dots ~,~ \alpha _{\theta (m-1)+1}\dots \alpha _{\theta (m)} ), 
 \end{equation*}
where the expression $z\alpha $ for $z\in \bbP ^{r-1}$ and $\alpha \in \GL _r$ denotes the right-action of matrices on row vectors.
Write also $[\alpha ]\colon \Psp \to \Psp $ for this action, so that $[\alpha \beta ]=[\beta ][\alpha ]$ holds, and $[\alpha \beta ]^*=[\alpha ]^*[\beta ]^*$ for pull-back operations.

Denote by $p\colon \bbP (E\GL _r)\to B\GL _r$ the 
projection.
It is a projective bundle with fiber $\bbP ^{r-1}$ and defines a projective bundle (= a presheaf locally isomorphic to the constant presheaf $\bbP ^{r-1}$) on the category $\MSm ^*/B\GL _r$, which motivates the following definition.

\begin{definition}
For each $i\ge 0$, we define the presheaf $p_*z^i_\rel $ of complexes on $\MSm ^*/B\GL _r$ as follows:
For each object $((X,D),\lambda ,f;n,\alpha )\in \MSm ^*/B\GL _r$,
denote by 
\begin{equation*}
p_*z^i_\rel ((X,D),\lambda ,f;n,\alpha ;\bullet )\subset z^i(\bbP ^{r-1}\times X |  \bbP ^{r-1}\times D,\bullet )
 \end{equation*}
the subcomplex of cycles $V\in z^i(\bbP ^{r-1}\times X |  \bbP ^{r-1}\times D,m)$ such that for every morphism
$(g,\varphi )\colon ((X',D')\lambda ',f')\to ((X,D), \lambda ,f) $ in
$\MSm ^*$,
its pull-back $(\id _{\bbP ^{r-1}}\times g)^*V\in z^i(\bbP^{r-1}\times X' |  \bbP ^{r-1}\times D',m)$ is well-defined. (This does not depend on the data $(n,\alpha )$.)

Given a morphism $(g,\varphi ,\theta )\colon ((X',D'),\lambda ',f';n',\alpha ')\to ((X,D),\lambda ,f;n,\alpha )$ in $\MSm ^*/B\GL _r$, define the pull-back map
\begin{equation*}
(g,\varphi ,\theta )^*\colon p_*z^i_\rel ((X,D),\lambda ,f;n,\alpha ;\bullet )
\to p_*z^i_\rel ((X',D'),\lambda ',f';n',\alpha ';\bullet )
 \end{equation*}
to be the pull-back along the morphism
(which depends on the data $(n,\alpha )$):
\begin{equation*}\begin{array}{ccccccc}
\bbP ^{r-1}\times X' &\xrightarrow{\sim } & \bbP ^{r-1}\times X' &\xrightarrow{\id \times g} &\bbP^{r-1}\times X. \\
(z,x') &\mapsto &(z\alpha _{1}\cdots \alpha _{\theta (0)},x') &
\end{array} \end{equation*}
Lastly, denote by $p^*\colon z^i_\rel \to p_*z^i_\rel $ the map of presheaves on $\MSm ^*/B\GL _r $ given by the pull-back along the projections $\bbP ^{r-1}\times X\to X $. 
\end{definition}

The presence of $p_*z^i_\rel $ is the main reason why we work in $\MSm ^*/B\GL _r$ rather than $\MSm ^*$.
This presheaf has the information of $p_*z^i_{\bbP (E)|\bbP (E) \times _X D}$ for all objects $((X,D),\lambda ,f)\in \MSm ^*$ and vector bundles $E\to X$ in the following sense.

Let us suppose $\Lambda = \{ *\} $ for simplicity, so that $\MSm ^*\simeq (X,D)_\Nis $.
Given a vector bundle $E\to X$ of rank $r\ge 0$, let $p\colon \bbP (E)\to X$ be the associated projective bundle $\bbP (E):=\mbf{Proj} (\operatorname{Sym}_{\mcal{O}_X} E^\vee )$.
We can consider the cycle complex $z^i_{\bbP (E)|\bbP (E)\times _X D}\colon (U\xrightarrow{\text{\'et}} \bbP (E) )$ $\mapsto z^i(U|U\times _XD ,\bullet )$ on $(\bbP (E)|\bbP (E)\times _X D)_\Nis $ and
its push-forward $p_*z^i_{\bbP (E)|\bbP (E)\times _X D}$ 
to $(X,D)_\Nis $ on the one hand.

On the other hand,
if we choose an open covering $X=\bigcup _{j} X_j$ and trivialization $\phi _j\colon E_{|X_j}\cong \mcal{O}_{X_j}^r$, we get a morphism of simplicial schemes
from the \v{C}ech construction
$\phi \colon \check{C}(\{ X_j\} _j) \to B\GL _r $ 
hence a simplicial object in $(X,D)_\Nis /B\GL _r$ (here, we give $X_j$ the divisor $D_j:=D\times _X X_j$). Therefore we can consider the restriction of $p_*z^i_\rel $ to $\check{C}(\{ X_j \} _j) _\Nis $.

One verifies that these two presheaves are canonically isomorphic on $\check{C}(\{ X_j \} _j) _\Nis $.
%

\begin{remark}\label{RemNonModulus}
The following non-modulus version
\begin{equation*}
p_*z^i \colon ((X,D),\lambda ,f;n,\alpha )
\mapsto z^i(\bbP ^{r-1}\times X,\bullet ),
 \end{equation*}
with the same ``presheaf'' structure as $p_*z^i_\rel $, plays a minor role later on.
Note that since we are not assuming the smoothness of $X$ along $D$ for objects of $\MSm $, it is difficult to make $p_*z^i$ functorial in a sensible way.
However, this does not pose a real problem because $p_*z^i$ is only used as a conceptual aid in some constructions and 
can be avoided if one is willing to write down all the raw data instead.

\end{remark}

\subsection{Line bundles and codimension $1$ cycles}\label{general-procedure}

We know that codimension $1$ cycles are closely related to line bundes. 
The following version of this relationship is repeatedly used in this article.
Below, we adopt the convention $\square ^n:= (\bbP ^1\setminus \{ \infty\} )^n=\Spec (k[t_1,\dots ,t_n])$ rather than $(\bbP ^1\setminus \{ 1\})^n$ used in \cite{BS}.

Let $X_\bullet $ be a semi-simplicial scheme with flat face maps (so that cycles and non zero-divisors can always be pulled back).
Let $L$ be a line bundle on it, i.e., the data of line bundles $L_n$ on $X_n$ for $n\ge 0$ and isomorphisms $d_i^*L_{n-1}\cong L_n$ for face maps $d_i\colon X_n\to X_{n-1}$, compatible with each other in an obvious sense.
Suppose we are given a section $\sigma \in \Gamma (X_0,L_0)$ which is everywhere a non zero-divisor.
We are going to define sections
\begin{equation*}
F_n\super{\sigma }=F_n\super{\sigma }(\seqdots{1}{n}{t})\in
\Gamma (X_n,L_n)\otimes _k k[t_1,\dots ,t_n]
 \end{equation*}
on $X_n\times \square ^n$ which are non zero-divisors everywhere.

Let us denote by $v^{[n]}_k\colon [0]\to [n]$ the inclusion of the $k$-th vertex.
Note that the groups $\Gamma (X_m,L_m)\otimes _k k[\seqdots{1}{n}{t}]$ are semi-cosimplicial in $m$ and cubical in $n$.
In particular, we have sections
$(v^{[n]}_k)_* \sigma \in \Gamma (X_n,L_n).$
We define $F\super{\sigma }_n$ by the formula 
\begin{equation*}
F_n\super{\sigma }(\seqdots{1}{n}{t}):=
\sum _{k=0}^n\left( (v_k^{[n]})_*\sigma \otimes t_k \prod _{\ell =k+1}^n (1-t_\ell ) \right)
 \end{equation*}
where $t_0=1$ by convention.
Of course, it is the map corresponding to the composite of a map
$\displaystyle \square ^n\to \Delta ^n$
from the $n$-cube to the (algebraic) $n$-simplex,
followed by 
the affine map $\Delta ^n\to \Gamma (X_n,L_n)$
sending the $k$-th vertex to $(\vertex{k}{n})_*\sigma $.
\footnote{In view of this, it is probably possible to carry out the construction in this paper on the (obvious) simplicial version of the cycle complex with modulus or even on the cycle complex as a simplicial abelian group.}

Recall that for $1\le j\le n$ and $\epsilon =0,1$, the map $\partial _{j,\epsilon }\colon \square ^{n-1}\to \square ^n$
is the embedding of the face $\{ t_j=\epsilon \} $:
$(\seqdots{1}{n-1}{t})\mapsto (t_1,\dots ,t_{j-1}, \epsilon ,t_j, \dots , t_{ n-1 })$.
Degenerate elements of a cubical set mean elements obtained by pull-back along degeneracy maps $\square ^n\to \square ^{n-1}$.
Also for $0\le i\le n$, denote by $d_i\colon [n-1]\to [n]$ the $i$-th face of $[n]$.
It is routine to check the following relations.

\begin{lemma}\label{LemRecord}
%
We have equalities in $\Gamma (X_n,L_n)\otimes _k k[\seqdots{1}{n-1}{t}]$:
\begin{equation*}
d_i^* F_{n-1}\super{\sigma } = 
\begin{cases}
\partial _{1,1}^*F_n\super{\sigma } &\text{ if } i=0, \\
\partial _{i,0}^{*}F_n\super{\sigma } &\text{ if }1\le i\le n.
\end{cases}
\end{equation*}
Also, the functions $\partial ^*_{j,1}F_n\super{\sigma }$ are degenerate for $1<j\le n$.
%
\end{lemma}
Let $z^1(-,\bullet )$ be Bloch's codimension $1$ cycle complex, which is a presheaf on the category of schemes and flat morphisms.
For a scheme $X$, let $\bbZ [X]$ be the additive presheaf generated by the presheaf of sets represented by $X$.
Lemma \ref{LemRecord} implies that the set of cycles $\{ \Gamma _n\super{\sigma } \} _n := \{ \Div (F_n\super{\sigma } ) \} _n$ determines a map of presheaves of complexes
\begin{equation*}
\Gamma \super{\sigma }\colon \bbZ [X_\bullet ] \to z^1(-,\bullet )
 \end{equation*}
on the small flat site over the semi-simplicial scheme $X_\bullet $.

Let $\sigma '\in \Gamma (X_0,L_0)$ be another nowhere zero-divisor. We set
\begin{multline}\label{EqHomotopy}
F _n\super{\sigma ,\sigma '}:= t_{1} F_n\super{\sigma '}(\seqdots{2}{n+1}{t}) 
+(1-t_{1})F_n\super{\sigma }(\seqdots{2}{n+1}{t})
\\{}
\in \Gamma (X_n,L_n)\otimes _k k[\seqdots{1}{n+1}{t}].
 \end{multline}
The cycles $\Gamma _n\super{\sigma ,\sigma '}:=\Div (F_n\super{\sigma ,\sigma '})$
will serve as a homotopy of $\Gamma \super{\sigma }$ and $\Gamma \super{\sigma '}$. 

\subsubsection{Variant}

Occasionally our line bundle $L$ will be 
given as the difference $L=L^+ \otimes (L^-)^\vee $ of two line bundles,
each equipped with a nowhere zero-divisor $\sigma ^\pm \in \Gamma (X_0,L_0^\pm )$ in degree $0$.
The construction of $\Gamma \super{\sigma }$ makes sense for the ratio $\sigma = \sigma ^+/\sigma ^- $.
In this case $F_n\super{\sigma }$
is the ratio of an element in
$\Gamma (X_n,L_n^+ \otimes (L_n^-)^{\otimes n})
\otimes _k k[\seqdots{1}{n}{t}]$
and an element in
$\Gamma (X_n,(L_n^-)^{\otimes n+1})$ (for example, $F_0\super{\sigma }=\sigma ^+/\sigma ^-$ is the ratio of an element in $\Gamma (X_0,L_0^+)$ and one in $\Gamma (X_0,L_0^-)$).

Also if a second presentation $L=L^{\prime +}\otimes (L^{\prime -})^\vee $ and non zero-divisors $\sigma ^{\prime \pm }\in \Gamma (X_0,L^{\prime \pm }_0)$ are given,
the construction of the homotopy $F\super{\sigma ,\sigma '}$ makes sense to give
$F_n\super{\sigma ,\sigma '}$ as the ratio of an element in
$\Gamma (X_n, L\otimes (L^-)^{\otimes n+1} \otimes (L^{\prime -})^{\otimes n+1})\otimes _k k[\seqdots{1}{n+1}{t}]$
and an element in
$\Gamma (X_n, (L^-)^{\otimes n+1} \otimes (L^{\prime -})^{\otimes n+1})$.

We will need the next complete intersection criterion. 
For the proof, the reader is referred to Lemma \ref{Lem:ProperInt}.

\begin{lemma}\label{Lem:well-def2}
Let $L^{(1)\pm }, \dots , L^{(i)\pm }$ be $2i$ line bundles on $X_\bullet $
equipped with sections
$\sigma ^\pm _a $ of $L_0^{(a)\pm }$
which are non zero-divisors ($1\le a\le i$).
Suppose the sequence of $i$ sections
\begin{equation*}
(\vertex{k_1}{n})_* \sigma ^\pm _1 , \dots ,
(\vertex{k_i}{n})_* \sigma ^\pm _i
 \end{equation*}
is a regular sequence
for every $n\ge 0$ and every
choice of $0\le k_1\le \dots \le k_i \le n$ and signs $\pm $.
Then the cup product 
\[ \Gamma  \super{\sigma _{1}} \cupp \dots \cupp \Gamma  \super{\sigma _{i}} \colon \bbZ [X_\bullet ]\to z^i(-,\bullet ) \]
is well-defined.

\end{lemma}

Here, the cup product is defined by the concrete
formulas in Appendix \ref{Appendix:cup}.
It is said to be well-defined if all the intersection products appearing in the expression are.

\subsection{The projective bundle formula}\label{SubsecProjBundle}

Now we can state and prove the main result of this section:
\begin{theorem}{\upshape (Projective bundle formula)}
\label{ProjBund}
For every $i\ge 0$, we have a canonical isomorphism in 
the Ninevich-local derived category
$D(\MSm ^*/\forget B\GL _r)$:
\[ \bigoplus _{j=0}^{r-1}z^{i-j}_\rel \xrightarrow[\sim ]{p^*(-) \cupp \xi ^j} 
p_*z^i_\rel  .\]
\end{theorem}

\label{ConsMaps}

First, we have to construct the maps. 
The following is a consequence of the Friedlander-Lawson moving lemma \cite[Theorem 3.1]{FL98}. 
In the lemma and onwards, the superscript $(-)^\circ $ will be used to indicate that some moving procedure is involved.

\begin{lemma}\label{Friedlander-Lawson}
Let $k$ be a field and $e\ge 1$ be an integer. Let $\bbP ^m$ be the $m$-dimensional projective space over $k$ ($m\ge 0$).
Then there is a codimension $1$ cycle $H^\circ $ on $\bbP ^m$ representing $\mcal{O}(1)$ such that for every effective cycle $Z\subset \bbP ^m_{\bar{k}}$ of positive dimension and of degree $\le e$ (over $\bar{k}$), the intersection of
$Z$ and 
$H_{\bar{k}}^\circ \text{ in }\bbP^m_{\bar{k}}$ 
is proper.
\end{lemma}


Let $H\supero{1}:=\Div ( \sigma _1 ) \subset \bbP ^{r-1}$ be any hyperplane (i.e., $\sigma _1\in \mcal{O}(1)$). Applying Lemma with $e=1$ gives a codimension $1$ cycle $H\supero{2}$ on $\bbP ^{r-1}$ which intersect every linear translate of $H\supero{1}$ properly. 
$H\supero{2}$ is written as the difference of the divisor of a section $\sigma _2^+\in \mcal{O}(d+1)$ of some degree $d>0$ and that of $\sigma _2^-\in \mcal{O}(d)$. (Of course, any sections with $d\ge 2$ work for this first step without appealing to Lemma.)
Next, apply Lemma \ref{Friedlander-Lawson} with $e=d+1$ and find sections $\sigma _3^+\in \mcal{O}(d'+1)$, $\sigma _3^- \in \mcal{O}(d')$ of some large degrees so that the cycle $H\supero{3}:=\Div (\sigma _3^+/\sigma _3^-)$ satisfies the condition that for every $(\alpha _1,\alpha _2)\in \GL _r(\bar{k})^2$ the following set has codimension $3$ in $\bbP ^{r-1}_{\bar{k}}$ (the symbol $|-|$ denotes the support of a cycle): 
\begin{equation*}
[\alpha _1]^* |H\supero{1}| \ \cap \ [\alpha _2]^*|H\supero{2}| \ \cap \ |H\supero{3}|.
 \end{equation*}
This works because the intersection of the first two factors (with the reduced structure) has been assured to have codimension 2 and has degree $\le 1\cdot (d+1) $.
Next apply Lemma \ref{Friedlander-Lawson} with $e=(d+1)(d'+1)$ to get $H\supero{4}=\Div (\sigma _4^+/\sigma _4^-)$ and so on.

In the end, we get
codimension $1$ cycles $H\supero{a}=\Div (\sigma _a^+/\sigma _a^-)$
on $\Psp $ ($1\le a\le r-1$) with the property that for every
$(\alpha _1,\dots ,\alpha _{r-1})\in \GL _r(\bar{k})^{r-1}$
the intersection
\begin{equation}\label{ConsequenceFL} [\alpha _1]^*|H\supero{1}|\cap \dots \cap [\alpha _{r-1}]^*|H\supero{r-1}| \quad \text{ in }\Psp _{\bar{k}}   \end{equation}
is zero-dimensional.
In other words, the sequence $ [\alpha _1]^*\sigma _1^\pm ,\dots ,[\alpha _{r-1}]^*\sigma _{r-1}^\pm $ is a regular sequence for every possible choice of signs $\pm $.
Applying the procedure in \S \ref{general-procedure} to these data, we get cycles which we denote by $\Gamma _n \supero{a}$ ($1\le a\le r-1$):
\[ \Gamma _n \supero{a}:=\Gamma \super{\sigma _a}_n 
\in  z^1(\bbP (E_n\GL _r),n). \]

Now, let us note that giving a section $((X,D),\lambda ,f)\xrightarrow{\alpha } B_n\GL _r$ of $B\GL _r$ in degree $n$ is equivalent to giving a map
$((X,D),\lambda ,f)\times \Delta ^n \to B\GL _r$ of simplicial presheaves on $\MSm $ (here, $\Delta ^n$ is a simplicial set, not a scheme).
This motivates the following:

\begin{definition}
Let $\Delta $ be the simplicial presheaf on $\MSm ^*/B\GL _r$ given by
\[ ((X,D),\lambda ,f;n,\alpha ) \mapsto \Delta ^n \]
on objects and $(g,\varphi ,\theta )\mapsto \theta $ on morphisms.
\end{definition}

The projection $\Delta \to *$ to the singleton is a sectionwise weak equivalence of simplicial presheaves on $\MSm ^*/B\GL _r$ because $\Delta ^n$ is contractible. 

\begin{definition}\label{DefGammaoProj}
For every object $\mfr{X}:=((X,D),\lambda ,f;n,\alpha )\in \MSm ^*/B\GL _r$, a simplex $\theta \in \Delta ^n_m$ and an index $a\in \{ 1,\dots ,r-1 \} $,
denote by 
\begin{equation*}
\Gamma _m\supero{a}(\mfr{X},\theta ) \in (p_*z^1)(\mfr{X},m)= z^1(\bbP ^{r-1}\times X,m )
 \end{equation*}
the pull-back of $\Gamma _m\supero{a}$ by the map
$\{ \bbP (E\GL _r)(\theta )\times \id _{\square ^m} \} \circ \{ \id _{\Psp }\times \alpha \times \id _{\square ^m}\} $:
\[ \begin{array}{rll}  \Psp \times X\times \square ^m\xrightarrow{\alpha } 
& \Psp \times (\GL _r)^n \times \square ^m \\
{} &=\bbP (E_n\GL _r)\times \square ^m \xrightarrow[\theta ]{} &
\bbP (E_m\GL _r)\times \square ^m . \end{array} \]
Using them, we define a map of complexes 
\begin{equation*}
\Gamma \supero{a}\colon \bbZ \otimes \Delta \to \ps z^1 \hspace{30pt}\text{ on }\MSm ^*/B\GL _r
 \end{equation*}
as follows:
On an object $\mfr{X}$ as above and in degree $m$, we must 
give a map of presheaves $\bbZ \otimes \Delta ^n_m \to p_*z^1(\mfr{X},m)=z^1(\bbP ^{r-1}\times X,m)$.
We do this by mapping $\theta \in \Delta ^n_m$ to $\Gamma _m\supero{a}(\mfr{X},\theta ) $.
\end{definition}
Of course, we could have pulled back the function $F_m\super{\sigma _a^+/\sigma _a^-}$ to define a function $F_m\supero{a}(\mfr{X},\theta )$ (a ratio of nowhere zero-divisors by direct inspection) and set 
$\Gamma _m\supero{a}(\mfr{X},\theta )$ as its divisor. Both give the same cycle.

The cup product below is well-defined for every $1\le j\le r-1$:
\[ C^{j\circ }:=\Gamma \supero{1}\cupp \dots \cupp \Gamma\supero{j} \colon \bbZ \otimes \Delta \to \ps z^j \]
thanks to proper intersection \eqref{ConsequenceFL} and Lemma \ref{Lem:well-def2} applied on each object of $\MSm ^*/B\GL _r$.
We tensor both sides with $z^{i-j}_\rel $ and apply the intersection product in the Nisnevich-local derived category (Appendix \ref{interProd}):
\[  z^{i-j}_\rel \otimes \Delta \xrightarrow{\id \otimes C^{j\circ } } z^{i-j}_\rel \otimes p_*z^j 
\xrightarrow{p^*(-)\cupp (-)}
p_*z^i_\rel .  \]
Composed with the inverse of the quasi-isomorphism
$z^{i-j}_\rel \otimes \Delta \xrightarrow{\sim }z^{i-j}_\rel $,
it gives us the maps which we call $p^*(-)\cupp \xi ^j $:
\[p^*(-)\cupp \xi ^j \colon\quad z^{i-j}_\rel \to \ps z^i_\rel 
\quad \text{ in }D(\MSm ^*/B\GL _r).   \]


\begin{proof}[Proof of Theorem \ref{ProjBund}]
We now claim that the map
$ \sum _{j=0}^{r-1}\bigl( p^*(-)\cupp \xi ^j\bigr)$ 
is an isomorphism in $D(\MSm ^*/B\GL _r)$.
For this purpose, we may work locally. We consider an object $\mfr{X}=((X,D),\lambda ,f;n,\alpha )$ and assume $X$ is henselian. 
Consider the weak equivalence
$z^{i-j}_{\rel }({\mfr{X}},\bullet )\hookrightarrow z^{i-j}_{\rel }({\mfr{X}},\bullet ) \otimes \Delta ^n$ 
corresponding to the inclusion of the $0$-th vertex $*\hookrightarrow \Delta ^n$. 
One computes the composition of it with $ p^*(-)\cupp \xi ^j $
as 
$V \mapsto
(H\supero{1}\cupp \dots \cupp H\supero{j})\times  V   
$
which is well-defined for all $V$.
This gives maps $\CH ^{i-j}(X|D,m)$ $\to $ $\CH ^i(\bbP ^{r-1}\times X |  \bbP ^{r-1}\times D,m)$
on homology.

Projective bundle formula for the higher Chow groups with modulus 
is known by Krishna, Levine and Park \cite[Th.5.6]{KL}, \cite[Th.4.6]{KP}.
For pairs $(X,D)$ with $X$ henselian, their and our maps
$\CH ^{i-j}(X|D,m)$ $\to $ $\CH ^i(\bbP ^{r-1}\times X |  \bbP ^{r-1}\times D,m)$
coincide, because both are computed as the classical intersection product once we are reduced to the proper intersection case.
This completes the proof of Theorem \ref{ProjBund}.
\end{proof}

%% file: UnivChernNMJv2.tex
\section{The universal Chern classes}\label{UnivChern}

A key ingredient of Chern classes in Bloch's higher Chow groups is the $r$-th power 
\begin{equation*}
\xi ^r\in \CH ^r(\bbP (E\GL _r)) \end{equation*}
of the class $\xi := [\mcal{O}(1)]$ on the simplicial scheme $\bbP (E\GL _r)$. 
By local homotopy theory, 
it corresponds to a map
$\xi ^r\colon \bbZ \to p_*z^r $ in $D(\Sm ^*/B\GL _r)$,
where $\Sm ^*$ is the non-modulus version of $\MSm ^*$.
In the first half of this section (\S\S \ref{SecStandardHyperplane}--\ref{SecSpecialization}), 
we construct a map
\begin{equation*}
\xi ^r_\rel \colon \bbZ \to p_*z^r_\rel \quad \text{ in }D(\MSm ^*/\Volo _r^\rel ) .
 \end{equation*}
Here $\Volo _r^\rel \subset B\GL _r$ is a subpresheaf called the relative Volodin space, 
which is known to represent the relative $K$-theory upon $\bbZ $-completion (\S\S \ref{SecVolodin}, \ref{ChernClassKrel}).
This map is a lifting of the classical $\xi ^r$ in the following sense.
In the underlying datum $F\colon \Lambda \to \MSm $, let $\Lambda _\emptyset \subset \Lambda $ be the full subcategory of objects $\lambda $ such that $D_\lambda =\emptyset $. 
Then $F$ restricts to $F_\emptyset \colon \Lambda _\emptyset \to \Sm $; $\lambda \mapsto X_\lambda $,
so we define $\Sm ^*:= \Sm /F_\emptyset $.
If one remembers (or {\em interprets}) how to construct the classical $\xi ^r$ correctly, it turns out that our $\xi ^r_\rel $ lifts $\xi ^r $ via the restriction map
\begin{equation*}
\Hom _{D(\MSm ^*/\Volo _r^\rel )}(\bbZ ,p_*z^r_\rel )
\to 
\Hom _{D(\Sm ^*/B\GL _r )}(\bbZ ,p_*z^r ).
 \end{equation*}


In the latter half (\S\S \ref{ChernClassXrel}--\ref{ChernClassKrel}), we discuss the stabilization process $r\to \infty $ and $\bbZ $-completion to derive Chern class maps 
\begin{equation*}
\msf{C}_{n,i}\colon K_n(X,D)\to H^{-n}_\Nis ((X,D) ,z^i_\rel ). 
 \end{equation*}

\subsection{Hyperplanes}\label{SecStandardHyperplane}

We often consider data $(\mfr{X},\theta )$ of an object $\mfr{X}=((X,D),\lambda ,f;n,\alpha )$ $\in \MSm ^*/B\GL _r$ and an ordered map $\theta \colon [m]\to [n]$.

\subsubsection{The standard hyperplanes}

Let $\seqdots{1}{r}{T}\in \Gamma (\bbP ^{r-1},\mcal{O}(1))$ be the homogeneous coordinates on $\bbP ^{r-1}$.
We apply \S \ref{general-procedure} to the line bundle $\mcal{O}(1)$ on $\bbP (E\GL _r) $ and sections $T_a$ for $a\in \{ 1,\dots ,r \}$ to get functions
\begin{equation*}
F_n\supers{a} := F_n\super{T_a} \in \Gamma (\bbP (E_n\GL _r),\mcal{O}(1))\otimes _k k[\seqdots{1}{n}{t}]
 \end{equation*}
and their divisors $\Gamma _n\supers{a}:= \Div (F_n\supers{a})$
$\in z^1(\bbP (E\GL _r),n)$.
Here we use superscripts $(-)^*$ because they will be useful only over certain open subsets which will be indicated by the same superscripts.

Repeat the construction of Definition \ref{DefGammaoProj} on these data to define cycles
$\Gamma \supers{a}_m(\mfr{X},\theta )\in z^1(\bbP ^{r-1}\times X,m)$.
%
They determine a map of complexes 
\begin{equation}\label{EqGammas}
\Gamma _{\MSm }\supers{a}\colon \bbZ \otimes \Delta \to p_*z^1\quad \text{ on }\MSm ^*/B\GL _r.
 \end{equation}

\begin{remark}
One may want to construct $\xi ^r_\rel $ as the cup product
$\Gamma \supers{1}_{\MSm }\cupp \dots \cupp \Gamma \supers{r}_{\MSm }$: $\bbZ \otimes \Delta \to \ps z^r$.
But the cup product is not well-defined due to the failure of proper intersection.
The easiest such example would be, taking $\Lambda =\{ *\}$: $r=2$, $\mfr{X}=((X,\emptyset );$ $n=1, \alpha =\begin{pmatrix}0 &1 \\ 1 &0 \end{pmatrix}\in (\Volo _2^\rel )_1)$, $\theta =\mrm{id}_{[1]}$. In this case the cycle that is supposed to represent the $(\mfr{X},\theta ) $-component of $\Gamma \supers{1}_{\MSm }\cupp \Gamma \supers{2}_{\MSm }$ always does not satisfy the face condition.
This is why we need the following constructions.

\end{remark}

\subsubsection{The generic hyperplanes}

Let $\{ x_{ab}\} _{1\le a,b\le r}$ be the coordinates of $\GL _r$, so that its function field is $k(\GL _r)=k(\{ x_{ab} \} _{a,b})$.
For each $a\in \{ 1,\dots ,r\} $, let us consider the ``generic translation'' of the coordinates:
$ T^\circ _a:= \sum _{b=1}^r T_b x_{ba}\in \Gamma (\Psp _{k(\GL _r)},\Oone ) . $ 
As in \S \ref{SubsecProjBundle}, we are using the superscript $(-)^\circ $ to indicate the involvement of moving procedure.
Applying the construction in \S \ref{general-procedure} 
and Definition \ref{DefGammaoProj} on $T^\circ _a$, we define cycles:
\begin{equation*}
\Gamma _{m, k(\GL _r)}\supero{a}(\mfr{X},\theta ) 
\in z^1((\bbP ^{r-1}\times X)_{k(\GL _r)},m).
 \end{equation*}
These objects are different from those denoted by similar symbols in Definition \ref{DefGammaoProj}, but since the older ones are not going to be used again in this section, there is no risk of confusion.

\subsubsection{Homotopy of the hyperplanes}

The homotopy in equation \eqref{EqHomotopy} gives us cycles
\begin{equation*}
\Gamma \superos{a}_{n,k(\GL _r)}:= 
\Gamma \super{T^\circ _a,T_a}_{n}
\in z^1(\bbP (E_n\GL _r)_{k(\GL _r)},n+1)
 \end{equation*}
and by pull-back as in 
Definition \ref{DefGammaoProj}, we define
\begin{equation*}
\Gamma _{m,{k(\GL _r)}}\superos{a}(\mfr{X},\theta ) 
\in z^1(\bbP ^{r-1}\times X_{k(\GL _r)} , m+1).
 \end{equation*}

\begin{definition}
For a field extension $L/k$, denote by $p_*z^i_L$ (the {\em non-modulus cycle complex with scalar extension}) the association of complexes
$\mfr{X}\mapsto p_*z(\mfr{X}\otimes _k L)$
for objects of $\MSm ^*/B\GL _r$.
There is an obvious scalar extension map $p_*z^i \to p_*z^i_L$.
\end{definition}
As in Remark \ref{RemNonModulus}, this is not a presheaf on $\MSm ^*/B\GL _r $ unless one restricts to some nice subcategory.
The set of cycles $\{ \Gamma _{m,k(\GL _r)}\supero{a}(\mfr{X},\theta ) \} _{m,\mfr{X},\theta }$ determines a map of complexes
\begin{equation*}
\Gamma _{\MSm ,k(\GL _r)}\supero{a} \colon \bbZ \otimes \Delta \to p_*z^1_{k(\GL _r)} .
 \end{equation*}

\subsection{The open covering and the complex $\mcal{Z}$}\label{TheComplZ}

\begin{definition}\label{F_iota}
Let $S\subset [n]$ be a subset consisting of $m+1$ elements.
We denote the unique injection $[m]\hookrightarrow [n]$ into $S$ also by the same letter.
Let us denote by 
$\Gamma \supers{a}(S)$
the pull-back of 
$\Gamma \supers{a}_m$ by the map
\begin{equation*}
\bbP (E\GL _r)(S)\times \id _{\square ^m} :\quad 
\bbP (E_n\GL _r)\times \square ^m\to \bbP (E_m\GL _r)\times \square ^m.
 \end{equation*}
Of course, it is the divisor of the function $F\supers{a}(S)$ similarly defined.

\end{definition}

\begin{definition}\label{DefBGL*}
Let $B_n\GL _r^*$ be the following open subset of $B_n\GL _r$:
\[  B_n\GL _r^*:= B_n\GL _r \setminus p \left( \bigcup _{0\le k_1\le \dots \le k_r \le n} 
\Gamma \supers{1}(v_{k_1}^{[n]})  
\cap \dots \cap 
\Gamma \supers{r}( v_{k_r}^{[n]})  \right)   \]
where $p\colon \bbP (E_n\GL _r)=\bbP ^{r-1}\times B_n\GL _r\to B_n\GL _r$ is the second projection. 
\end{definition}
It follows that a point $\alpha =(\seqdots{1}{n}{\alpha })\in B_n\GL _r (k(\alpha ))$ is in $B_n\GL _r^*$ if and only if for every choice of $0\le k_1\le \dots \le k_r \le n$, the intersection in $\bbP ^{r-1}_{k(\alpha )}$
\[ [\alpha _{1}\alpha _2\cdots \alpha _{k_1}]^*\{ T_1=0\} \cap \dots 
\cap [\alpha _{1}\alpha _2\cdots \alpha _{k_r}]^*\{ T_r=0\}   \]
is empty.
Note that whenever $\seqdots{1}{n}{\alpha }$ are all upper triangular, the sequence $(\seqdots{1}{n}{\alpha })$ belongs to $B_n\GL _r^*$. 
The simplicial structure of $B\GL _r $ restricts to 
the schemes $\{ B_n\GL _r^* \} _n$. 
%

\begin{definition}
For a datum $(\mfr{X},\theta ) $ of an object $\mfr{X}\in \MSm ^*/B\GL _r $ and a map $\theta \colon [m]\to [n]$, define an open subset $X^*_{\alpha ,\theta }$ of $X$ by
$X^*_{\alpha ,\theta }:=(B\GL _r(\theta )\circ \alpha )^{-1}(B_n\GL _r^* ). $  
\end{definition}

\begin{definition}
Define the simplicial subpresheaves $\Delta ^* $ and $\Delta ^\circ $ of $\Delta $ on $\MSm ^*/B\GL _r$ by (where $\mfr{X}=((X,D),\lambda ,f;n,\alpha )$):
\begin{equation*}
\Delta ^*(\mfr{X})_m= \{ \theta \in \Delta ^n_m \mid X_{\alpha ,\theta }^*=X  \} ,
\qquad \Delta ^\circ (\mfr{X}) =\begin{cases} 
\Delta ^n &\text{ if }D=\emptyset ,\\
\emptyset &\text{ if }D\neq \emptyset .
\end{cases}
 \end{equation*}
Also, set $\Delta \os := \Delta ^\circ \cap \Delta ^*$.
\end{definition}
Let us say that a morphism of the form $(f,\id ,\id )\colon \mfr{X}'\to \mfr{X}$ in $\MSm ^*/B\GL _r$ is an open immersion if the underlying morphism $f\colon X'\to X$ is an open immersion and if $D'=X'\times _X D $. Then $\Delta ^\circ $ and $\Delta ^*$ are open subpresheaves of $\Delta $.
The map $\Delta ^\circ \sqcup \Delta ^*\to \Delta $ is not a surjection of Zariski sheaves, but becomes so on the smaller category $\MSm ^*/\Volo _r^\rel $ introduced in \S \ref{SecVolodin}.
Consequently, the complex $\mcal{Z}$ below becomes quasi-isomorphic to $\bbZ $ on that category.

\begin{definition}
Define the complex $\mcal{Z}$ on $\MSm ^*/B\GL _r$ by:
\[ \mcal{Z}:=\cone \left( \bbZ \otimes \Delta \os 
\xrightarrow[(\mrm{incl.},\mrm{incl.})]{}
(\bbZ \otimes \Delta ^\circ) \oplus (\bbZ \otimes \Delta ^*)\right) . \]
\end{definition}
The maps $\Gamma _{\MSm ,k(\GL _r)}\supero{a}$ on $\bbZ \otimes \Delta ^\circ $, $\Gamma _{\MSm }\supers{a}$ on $ \bbZ \otimes \Delta ^*$ and the homotopy $\{ \Gamma _{m,k(\GL _r)}\superos{a}  \} _m$ on $\bbZ \otimes \Delta \os $ determine a map of complexes
\begin{equation}\label{EqGamma}
\Gamma _{\MSm ,k(\GL _r)}\super{a}\colon \mcal{Z}\to p_*z^1_{k(\GL _r)}.
 \end{equation}
The next lemma follows from the definition of $B\GL _r^*$ and the algebraic independence of $\{ x_{ab} \} _{a,b}$.
\begin{lemma}\label{LemIntEmpty}
\leavevmode
\begin{enumerate}[{\upshape (i)}]
\item 
The intersection of
$\displaystyle \bigcap _{a=1}^r \Gamma _{0,k(\GL _r)} \supero{a}(\mfr{X} ,v_{\theta (k_a)}^{[n]} )  
$ 
and $(\bbP ^{r-1}\times X^\circ )_{k(\GL _r)}$ is empty for every choice of $0\le k_1\le \dots \le k_r \le n $.

\item
The intersection of
$\displaystyle  \bigcap _{a=1}^r \Gamma _{0} \supers{a}(\mfr{X} ,v_{\theta (k_a)}^{[n]} )
$ 
and $\bbP ^{r-1}\times X_{\alpha ,\theta }^* $ is empty for every choice of $0\le k_1\le \dots \le k_r \le n $.

\item
The intersection of $(\bbP ^{r-1}\times X_{\alpha ,\theta }\os )_{k(\GL _r)}$ and: 
\begin{equation*}
\left( \bigcap _{a=1}^b \Gamma _{0,k(\GL _r)} \supero{a}(\mfr{X} ,v_{\theta (k_a)}^{[n]} )  \right)
\cap
\left( \bigcap _{a=b+1}^r \Gamma _{0} \supers{a}(\mfr{X} ,v_{\theta (k_a)}^{[n]} )_{k(\GL _r)} \right)
 \end{equation*}
is empty for every choice of $0\le k_1\le \dots \le k_r \le n $
and $0\le b\le r $.
\end{enumerate}\end{lemma}
By Lemma \ref{LemIntEmpty}, we can apply Lemmas \ref{Lem:well-def2} and \ref{Lem:ProperInt2} to conclude that the cup product
\begin{equation}\label{NonModXi^r} \Gamma \super{1}_{\MSm }\cupp \dots \cupp \Gamma \super{r}_{\MSm }\colon \mcal{Z}\to \ps z^r_{k(\GL _r)}  
\end{equation}
is well-defined.
We now introduce a modulus version $p_*z^i_{\rel ,k(\GL _r)}\subset p_*z^i_{k(\GL _r)}$ of the target.
We shall show that the map \eqref{NonModXi^r} factors through it 
when restricted to the category $\MSm ^*/\Volo _r^\rel $ introduced in \S \ref{SecVolodin}.

\label{ModifOutside}

\begin{definition}
For a field extension $L/k$, define the presheaf of complexes $\ps z^i_{\rel ,L}$ ({\em cycle complex with modulus with scalar extension}) on $\MSm ^*/B\GL _r$ by the rule
\[ \mfr{X}=((X,D),\lambda ,f)\mapsto \begin{cases}\ps z^i_\rel (\mfr{X}_L ) &\text{ if }D=\emptyset , \\ \ps z^i_\rel (\mfr{X} )&\text{ if }D\neq \emptyset  \end{cases}  \]
with the same presheaf structure as 
$\ps z^i_\rel $.
There is a scalar extension map $\ps z^i_\rel \to \ps z^i_{\rel ,L}$.
\end{definition}
This definition may appear strange at first glance. It is motivated by the fact that Bloch's specialization map \S \ref{SecSpecialization} is available (and necessary) only when $D=\emptyset $.
%

\subsection{The relative Volodin space}\label{SecVolodin}

Now we introduce the relative Volodin space presheaf.
Its significance lies in its relation to the relative $K$-theory; see Theorem \ref{thm:Krel}.

\begin{definition}\label{DefVolodin}
Let $(X,D)\in \MSm $. Denote by $I=I_D\subset \mcal{O}_X$ the ideal sheaf defining $D$. 
Let $r\ge 0$ be a non-negative integer and $\sigma $ a partial order on the set $\{ 1,\dots ,r \} $.
Then the subgroup $T^\sigma (X,D) \subset \GL _r(X)$ is defined to be the set of matrices $(x_{ab})_{1\le a,b\le r}$ such that $x_{ab}\equiv \delta _{ab} \mod I_D$ (Kronecker's $\delta $) unless $a\overset{\sigma }{<} b$.
For example, if $\sigma $ is the usual total order on $\{ 1,2,3 \} $, then elements of $T^\sigma (X,D)$ look like:
\[  \begin{pmatrix} 1+I &\mcal{O}_X &\mcal{O}_X \\ I  & 1+I  & \mcal{O}_X  \\ I &I & 1+I \end{pmatrix} .  \]
If another order $\sigma '$ {\em extends} $\sigma $ (i.e.\ if $a\overset{\sigma }{<}b$ implies $a\overset{\sigma '}{<}b$), we have $T^\sigma (X,D)\subset T^{\sigma '}(X,D)$.
The {\em Volodin space} $\Volo _r(X,D)$ is the simplicial subset of $B\GL _r(X)$ defined by
\[ \Volo _r(X,D)=\bigcup _{\sigma }BT^\sigma (X,D) \subset B\GL _r(X).   \]
We set $\Volo (X,D)=\varinjlim _r \Volo _r(X,D)\subset B\GL (X)$. Define a Nisnevich sheaf $\Volo _r^\rel $ on $\MSm $ by $(X,D)\mapsto \Volo _r(X,D)$.
\end{definition}

We also denote by $\Volo _r^\rel $ the presheaf induced by the forgetful functor $\MSm ^*\to \MSm $. In particular, we have the site $\MSm ^*/\Volo _r^\rel $ fibered over it. 
The inclusion $\Volo _r^\rel \hookrightarrow B\GL _r$ induces a functor $\MSm ^*/\Volo _r^\rel \to \MSm ^*/B\GL _r$. Every presheaf and map of presheaves on the latter restrict to the former. It follows for example that the projective bundle formula in \S \ref{GlobalProjBund} holds on $\MSm ^*/\Volo _r^\rel $ in the same form.


\begin{lemma}\label{LemZarCov}
The map of simplicial presheaves $\Delta ^\circ \sqcup \Delta ^*\to \Delta $ on $\MSm ^*/\Volo _r^\rel $ is surjective in the Zariski topology.
\end{lemma}
\begin{proof}
It suffices to prove the following:
For each $(X,D)\in \MSm ^*$, $\alpha =(\alpha _1,\dots ,\alpha _n)\in B_n\GL _r (k(\alpha ))$ and $\theta \in \Delta ^n_m$, we have $X=(X\setminus D )\cup X^*_{\alpha ,\theta }$.

Let $\sigma $ be a (total) order on $\{ 1,\dots ,r \} $ such that $\alpha \in BT^\sigma (X,D)$. The matrices
$\seqdots{1}{n}{\alpha }$ are all upper triangular modulo $I_D$ up to permutation by $\sigma $.
It follows from the remark subsequent to Definition \ref{DefBGL*} that every $x\in D$ belongs to $X^*_{\alpha ,\theta }$. This proves the lemma.
\end{proof}

By Lemma \ref{LemZarCov}, the complex $\mcal{Z}$ is Zariski locally quasi-isomorphic to $\bbZ \otimes \Delta \simeq \bbZ $ by the map
$(\bbZ \otimes \Delta ^\circ) \oplus (\bbZ \otimes \Delta ^*) $
$\xrightarrow[\mrm{incl.}\sqcup (-\mrm{incl.})]{}$
$ \bbZ \otimes \Delta   $ when restricted to $\MSm ^*/\Volo _r^\rel $.


The rest of
this subsection is devoted to the proof of the following:
\begin{theorem}\label{ModulusVolodin}
The map \eqref{NonModXi^r} 
restricted to $\MSm ^*/\Volo _r^\rel $ factors through the subcomplex $\ps z^r_{\rel ,k(\GL _r)}$.
\end{theorem}
Note that the assertion only concerns the part
$\Gamma \supers{1}_{\MSm }\cupp \dots \cupp \Gamma \supers{r}_{\MSm }\colon \bbZ \otimes \Delta ^*\to \ps z^r $.
The following
criterion for the modulus condition will be useful.

\begin{definition}{\upshape (\cite[\S 4]{BS})}
Let $A$ be a commutative ring with unit and
$I$ be an ideal.
A polynomial
\begin{equation*}
f=\sum _{\seqdots{1}{n}{\lambda }} a_{\seqdots{1}{n}{\lambda }}t^{\lambda _1}\cdots t^{\lambda _n}
\in A[\seqdots{1}{n}{t}]
\end{equation*}
is said to be {\em admissible} 
if $a_{\seqdots{1}{n}{\lambda }}\in I^{\max _i \{ \lambda _i \}} $
and if $a_{0,\dots ,0}$ maps into $(A/I)^*$.

\end{definition}

\begin{lemma}{\upshape (\cite[Lemma 4.3]{BS})}
\label{BS-criterion}
Let $X$ be an affine scheme equipped with an effective Cartier divisor $D$.
Let $V$ be an integral closed subscheme of $X\times \square ^n$.
If the defining ideal for $V$ contains an admissible polynomial with respect to the defining ideal of $D$, then $V$ satisfies the modulus condition.
\end{lemma}

When $X$ is a $k$-scheme of finite type equipped with an ideal sheaf $I$,
let us say that an ideal sheaf ${J}$ on $X\times \square ^n$
is {\it admissible} if there exists an affine open covering $\{ U_\alpha  \} _\alpha $ of $X$ such that 
${J}$ restricted to each $U_\alpha \times \square ^n$ contains an admissible polynomial with respect to $I(U_\alpha )$.
Note that if $J$ is admissible and $f\colon X'\to X$ is a morphism from another scheme,
the ideal sheaf $(f\times \id _{\square ^n})^*J$ on $X'\times \square ^n$ is admissible with respect to $f^*I$,
because elements in a power $I^\lambda $ pull back into the power $(f^*I)^\lambda $.

\begin{notation}
Let $\{ x^i_{bc} \} ^{i\in \{ 1,\dots ,n \} } _{b,c\in \{ 1,\dots ,r  \} }$
be the coordinates for $B_n\GL _r=(\GL _r)^n$.
For an ordering $\sigma $ on $\{ 1,\dots ,r  \} $,
let $I^\sigma $ be the ideal of $\mcal{O}_{B_n\GL _r}$
generated by
$x^i_{bc}-\delta _{bc}$
with $i\in \{ 1,\dots ,n  \} $ and
$b,c\in \{ 1,\dots ,r \} $ such that $b\overset{\sigma }{ \not<  }  c$, where $\delta _{bc}$ is Kronecker's delta.

\end{notation}

Under this notation, a section $\alpha \in (\Volo _r^\rel )_n(X,D)$ is the same as a morphism of schemes
$X\to B_n\GL _r$ which maps the subscheme $D$ into the closed subscheme $V(I^\sigma )$ for some $\sigma $.
For subsets $S,T\subset [n]$, let us write $S\le T$ to mean $s\le t$ for all $s\in S$ and $ t\in T$.
Also, recall the symbol $F\supers{a}(S)$
from Definition \ref{F_iota}.

\begin{proof}[Proof of Theorem \ref{ModulusVolodin}]
The cycles defining the map 
$\Gamma \supers{1}_{\MSm }\cupp \dots \cupp \Gamma \supers{r}_{\MSm }$
are pull-backs of the universal cycles on $\bbP ^{r-1}\times B_n\GL _r^* \times \square ^n$
by individual maps $\bbP ^{r-1}\times X\times \square ^n\to \bbP ^{r-1}\times B_n\GL _r^*\times \square ^n$.
In view of Lemma \ref{BS-criterion} and this observation,
Theorem \ref{ModulusVolodin} follows from
the following lemma.

\begin{lemma}\label{Ideal-is-admissible}
Let $n\ge 0$ and $r\ge 1$ be integers and $\sigma $ an order on $\{ 1,\dots ,r \} $.
Let $S_1\le \dots \le S_r$ be non-empty subsets 
of $[n]$.
Then the ideal sheaf on ${\bbP ^{r-1}\times B_n\GL _r\times \square ^n}$
associated to the homogeneous ideal generated by: 
\begin{equation*}
F\supers{a}(S_a ) \quad 1\le a\le r
 \end{equation*}
is admissible with respect to the ideal sheaf $\mcal{O}_{\bbP ^{r-1}}\otimes _k I^\sigma $
on $\bbP ^{r-1}\times B_n\GL _r$.

\end{lemma}

Lemma \ref{Ideal-is-admissible} follows from a more precise claim below.
Note that we may obviously assume that $\sigma $ is a total order and,
by symmetry, that $\sigma $ is the usual order $\sigma =\{ 1<\dots <r  \} $.
Let us write $I:= I^\sigma $ for this $\sigma $.

For $S\subset \{ 1,\dots ,n  \} $, let us denote by
$[t_i\mid i\in S] \subset k[\seqdots{1}{n}{t}] $ 
the $2^{|S|}$-dimensional $k$-vector space spanned by
monomials $\prod _{i\in S}t_i^{\epsilon _i}$
where $\epsilon _i\in \{ 0,1 \} $.
To ease the notation, we shall use the phrase: 
\begin{quote}
\centering
``a polynomial of the form $T_c\cdot I\cdot [t_i\mid i\in S]$'' ($c\in \{ 1,\dots ,r \} $)
 \end{quote}
to mean a sum of polynomials of the form 
$T_c\cdot x \cdot f$ with $x \in I \subset k[x^i_{ab}]$ and $f\in [t_i\mid i\in S]$.
For a non-empty subset $S$ of $[n]$, we write $S'$ for the set $S\setminus \{ $the minimum element of $S \} $.

\begin{claim}\label{Claim-admissible}
For any $a\in \{ 1,\dots ,r\} $,
the ideal of the polynomial ring
\begin{equation*}
k[\seqdots{1}{r}{T}][x^i_{bc}\mid ^{ 1\le i\le n }_{  b,c\in \{ 1,\dots ,r  \}  }][\seqdots{1}{n}{t}] 
\end{equation*}
generated by
$\{ F\supers{b}(S_b) \} _{a\le b\le r}$ 
contains a polynomial of the form
\addtocounter{equation}{1} 
\newcounter{AdmissibleForm}\setcounter{AdmissibleForm}{\theequation} 
\begin{equation}
T_a+ \sum _{c=1}^r \left( T_c\cdot I\cdot \left[ t_i\mid i\in S_a' \cup S_{a+1}'\cup \dots \cup S_r'\right] \right) .
\tag{\theequation $_a$}
 \end{equation}
\end{claim}

Claim \ref{Claim-admissible} implies Lemma \ref{Ideal-is-admissible}
because formula (\theAdmissibleForm $_a$) divided by $T_a$ gives an admissible polynomial over the affine open set $\{ T_a \neq 0\} $.

{\em Proof of Claim}.
We proceed by descending induction on the index $a$ starting with $a=r$.
Let us write down the definition of $F\supers{a}(S)$, where $a\in \{ 1,\dots ,r \}$ and $S\subset [n]$ is a non-empty subset with $s+1$ elements:
\begin{align*}
F\supers{a}(S)
= &\sum _{i=1}^s\left( (S_*(v_i^{[s]})_* T_a )\cdot t_{S(i)}
\prod _{j=i+1}^s (1-t_{S(j)}) \right) \\
&+(S_*(v_0^{[s]})_* T_a )\cdot 
\prod _{j=1}^s (1-t_{S(j)}) .
\end{align*}
Since we are in the group $T(X,D)$ of upper triangular matrices modulo $I$,
for any map $v\colon [0]\to [n]$ (such as $S\circ v_i^{[s]}$)
the function
$v_*T_{ a }$
has the form
\begin{equation*}
\left( \sum _{b=1}^{a-1} T_b\cdot I \right) + T_a \cdot (1+I)
+ \left( \sum _{b=a+1}^r T_b \cdot \mcal{O} \right) ,
 \end{equation*}
where $\mcal{O}:= \mcal{O}_{B_n\GL _r}$.
By these two formulas and the identity
$t_{S(s)}+t_{S(s-1)}(1-t_{S(s)})+ \dots + (1-t_{S(1)})\cdots (1-t_{S(s)})=1$, 
we get:
\begin{equation*}\label{F-has}
F\supers{a}(S)=T_a+\sum _{b=1}^a \left( T_b\cdot I\cdot [t_i\mid i\in S'_a] \right)
+\sum _{b=a+1}^r \left( T_b\cdot \mcal{O}\cdot [t_i\mid i\in S'_a] \right) .
\end{equation*}
This already proves the assertion for $a=r$.

Now suppose $a< r $.
By descending induction, we know that the ideal in question contains polynomials of the form (\theAdmissibleForm $_b$)
for $b=a+1,\dots ,r$.
In particular, we get:
\begin{equation*}
\sum _{b=a+1}^r \left( T_b\cdot \mcal{O}\cdot [t_i\mid i\in S'_a] \right)
\equiv \sum _{c=1}^r \left( T_c\cdot I\cdot [t_i\mid i\in S'_a\cup S'_{a+1}\cup \dots \cup S'_r] \right)
 \end{equation*}
modulo the ideal in question.
Here we used the fact that the product of an element in
$[t_i\mid i\in S]$ and one in $[t_i\mid i\in T]$
with $S \cap T=\emptyset $ belongs to
$[t_i\mid i\in S\cup T ]$.
The last two formulas 
give a formula of the form (\theAdmissibleForm $_a$).
This completes the proof of Claim \ref{Claim-admissible}, hence also of Theorem \ref{ModulusVolodin}.
%
\end{proof}

Hence we have obtained a map
\begin{equation}\label{EqObtainedCupProduct}
\Gamma\super{1}_{\MSm }\cdot\ldots\cdot
\Gamma\super{r}_{\MSm }\colon
\mcal{Z}\to p_*z_{\rel ,k(\GL _r)} ^r 
 \end{equation}
in $D(\MSm ^* /\Volo _r^\rel )$.

\subsection{Specialization map, and end of construction of $\xi ^r_\rel $}\label{SecSpecialization}

Bloch defined a specialization map $z^i(X_{L},\bullet )\to z^i(X,\bullet )$ in the derived category when $L/k$ is a purely transcendental extension of finite degree equipped with a transcendence basis and $X$ is an equi-dimensional $k$-scheme \cite[pp.291, 292]{Bl86}.
Likewise, we can define a specialization map
\[
	\mrm{sp}_{L/k}\colon p_*z^i_{\rel ,L}\to p_*z^i_\rel 
\]
in $D(\MSm ^*/\Volo ^\rel _r)$ by using his map when $D=\emptyset $ and setting it to be the identity when $D\neq \emptyset $, roughly speaking.
See Appendix \ref{BlochSp} for a careful definition.

This applies in particular to the field $L=k(\GL _r)$.
Since the specialization map depends on the transcendental basis and the order thereof, we fix a total order on the set $\bb{N}\times \bb{N}$ once and for all, and use the induced order on the variables 
$\{ x_{ab} \} _{(a,b)\in \{ 1,\dots ,r \} ^2}$.

\begin{def-lem}\label{def:xi^r_rel}
We define the map $\xi ^r_\rel $ in $D(\MSm ^*/\Volo _r^\rel )$ by:
\begin{equation}\label{Xi^r!!}
\xi ^r_\rel \colon \bbZ \xleftarrow{\sim } \mcal{Z}\xrightarrow{\eqref{EqObtainedCupProduct}} p_*z^{r}_{\rel ,k(\GL _r)} \xrightarrow{\mathrm{sp}} p_*z^r_\rel   . 
\end{equation}
It does not depend on the choice of the ordering. 
\end{def-lem}
\begin{proof}
Suppose that two consecutive variables in a given order are interchanged. Via the corresponding automorphism on $k(\GL _r)$ and hence on $p_*z^r_{\rel ,k(\GL _r)}$, the problem is equivalent to the situation where
the specialization map stays the same but
the map $\mcal{Z}\to p_*z^r_{\rel ,k(\GL _r)}$ is constructed with the two variables interchanged.
But this difference is within homotopy by the homotopy at the end of \S \ref{general-procedure}.
\end{proof}

\begin{remark}\label{AltDef}
We will need to know that some cycles we have defined so far have certain alternative constructions when the base is restricted.
\begin{enumerate}[{\upshape (i)}]
\item \label{ItemT^o_a}
After the restriction of the base
$\Volo _r^\rel \times \Volo _s^\rel \hookrightarrow \Volo _{r+s}^\rel $,
the map $\xi ^{r+s}_\rel $ (on the level of presheaf map $  \mcal{Z}\to p_*z^{r+s}_{\rel ,k(\GL _{r+s})} $)
can be defined using the alternative $T_{a}^\circ $ as follows:
$T_{a}^\circ := 
\sum\limits _{b=1}^r T_{b} x_{ba} $ if $1\le a\le r$,
and $T_{a}^\circ :=
\sum\limits _{b=r}^{r+s} T_{b} x_{ba} $ if $r+1\le a\le r+s$.
This works because the proper intersection condition needed is now weaker (i.e., an analog of Lemma \ref{LemIntEmpty} is true with this $T^\circ _a $ on $\Volo _r^\rel \times \Volo _s^\rel $).
In this case, it is defined over the subfield $k(\GL _r\times \GL _s)$ of $k(\GL _{r+s})$ (the inclusion comes from the projection $M_{r+s}\to M_r\times M_s$ of the spaces of matrices).

\item 
In \S \ref{ConsMaps}, we defined maps (for $j\le r-1$)
\[ p^*(-)\cupp \xi ^j\colon  \quad z_\rel ^{i-j} \to  
p_*z^i_\rel \quad \text{ in }D(\MSm ^*/B\GL _r)  \]
using cycles given by the Friedlander-Lawson moving lemma.
On the smaller category $\MSm ^*/\Volo ^\rel_r$, it can be constructed in the style of this \S \ref{UnivChern}. 
Namely we use the maps
$\Gamma \super{a}_{\MSm }$ in formula \eqref{EqGamma} 
\[
C^j_{\MSm ^*/\Volo _r^\rel }:=\Gamma \super{1}_{\MSm }\cupp \dots \cupp \Gamma \super{j}_{\MSm }\colon \mcal{Z}\to p_*z^j_{k(\GL _r)} 
\]
(actually, any choice of $j$ members out of $\{ 1,\dots ,r \} $ will do, in place of $ 1,\dots ,j $),
and use the specialization map.
When $j=r$, this is the same as the construction of $\xi ^r_\rel $, so it is well-defined also for $j=r$.
Again, when we restrict the base from $\Volo _{r+s}^\rel $ to
$\Volo _r^\rel \times \Volo _s^\rel $ as in {\upshape \ref{ItemT^o_a}},
we can use the simpler choice of $T^\circ _a$ (which are built in $\Gamma \super{a}_{\MSm }$).

\end{enumerate}

\end{remark}

\subsection{Chern classes on the relative Volodin space}\label{ChernClassXrel}

Let $r>0$. 
In Theorem \ref{ProjBund}, we have proved an isomorphism
\[
	p^*(-)\cdot\xi^j \colon \bigoplus_{j=0}^{r-1} z^{r-j}_\rel \xrightarrow{\simeq} p_*z^r_\rel
\]
in $D(\MSm^*/B\GL_r)$, and thus in $D(\MSm^*/\mbf{X}_r^\rel)$.
In Definition-Lemma \ref{def:xi^r_rel}, we have constructed a map
\[
	\xi^r_\rel\colon \bb{Z} \to p_*z^r_\rel
\]
in $D(\MSm^*/\mbf{X}_r^\rel)$.
It follows that there are unique morphisms $c_i \colon \bb{Z} \to z^i_\rel$ in $D(\MSm^*/\mbf{X}_r^\rel)$ for $1\le i\le r$ which satisfy the equality of maps
$\bbZ \rightrightarrows p_*z^r_\rel $: 
\addtocounter{equation}{1} 
\newcounter{CharacterizingEq}\setcounter{CharacterizingEq}{\theequation} 
\begin{equation}
	\xi^r_\rel + (p^*c_1)\cdot \xi^{r-1} + \dotsb +p^*c_r = 0.
	\tag{\theequation $_r$}
 \end{equation}
It is convenient to define $c_i:= 0$ for $i>r$.
By Theorem \ref{lem:LocHtp} applied to $\mcal{C}= \MSm ^*$,
we have an isomorphism
\[
\Hom _{\Ho (s\PSh (\mcal{\MSm ^*}))}(\Volo _r^\rel ,K(z^i_\rel ,0))
\cong \Hom _{D (\MSm ^*/\Volo _r^\rel )}( \bbZ ,z^i_\rel ).
\]
Denote again by $c_i$ the corresponding map
$	c_i\colon \mbf{X}_r^\rel \to K(z^i_\rel,0) $
in the Nisnevich-local homotopy category of simplicial presheaves $\Ho(s\PSh (\MSm^*))$.

\begin{definition}\label{def:c_Vol}
The above-defined maps:
\begin{equation*}
c_i\colon \bbZ \to z^i_\rel 
\quad \text{\upshape or }\quad
c_i\colon \Volo _r^\rel \to z^i_\rel 
 \end{equation*}
are called the {\em Chern classes} (of rank $r$).
\end{definition}

Note that for every $r,i\ge 1$ the composite in $\Ho(s\PSh (\MSm ^*))$:
\[ *=(\text{the identity matrix})\hookrightarrow \Volo _r^\rel \xrightarrow{c_i} K(z^i_\rel ,0) \]
equals the constant map to the base point 
because the map $\xi ^r_\rel $ 
is represented by the empty cycle 
when restricted to $\MSm ^*/ \{ \id \} $.
By this fact and a somewhat standard result below,
it follows that
$c_i$ come from unique maps $\mbf{X}_r^\rel \to K(z^i_\rel,0)$ in $\Ho (s\PSh _*(\MSm ^*))$, the homotopy category of {\it pointed} simplicial presheaves.

\begin{lemma}{\upshape (cf.~\cite[Prop.5.2]{AS})}
\label{LemPointed}
Let $(X,x)$ be a pointed object in $s\PSh (\MSm ^*)$ and $(K,e_K)$ be a group object in the same category.
Then the square of sets below is cartesian:
\[ \xymatrix{
 \Hom _{\Ho (s\PSh _*(\MSm ^*))}((X,x),(K,e_K)) \ar[r] \ar[d]
 & {\mrm{pt}}
 \ar@{->}[d]
\\%
\Hom _{\Ho (s\PSh (\MSm ^*))}(X,K) 
 \ar[r]
&\Hom _{\Ho (s\PSh (\MSm ^*))}(x,K) ,
}   \]
where the left vertical arrow is the ``forget the base point'' map and the right vertical arrow maps the point to the constant map at $e_K$.
\end{lemma}


Next, associated with the embedding $\iota \colon \GL _r\hookrightarrow \GL _{r+1};~\alpha \mapsto \begin{pmatrix} \alpha &0 \\ 0& 1 \end{pmatrix}$, we have embeddings $\Volo _r^\rel  \hookrightarrow \Volo _{r+1}^\rel $
and $\bbP ^{r-1}\hookrightarrow \bbP ^{r}; (T_1:\dots :T_r)\mapsto (T_1:\dots :T_r:0)$.
The following is a special case of Proposition \ref{Prop:SecWhitney} below.

\begin{lemma}\label{CorStabMap}
The following diagram in $\Ho (s\PSh _*(\MSm ^*))$: 
\[ 
\xymatrix{
\Volo _r^\rel \ar[r]_\iota \ar@/_7mm/[rr]_{c_i} &
\Volo _{r+1}^\rel \ar[r]_(0.4){c_i} & K(z^i_{\rel },0).
}  \]
commutes for $i\ge 1$.
\end{lemma}

\subsection{Chern classes on the relative K-theory}\label{ChernClassKrel}

We let $\mrm{K}$ be a functorial model of Thomason-Trobaugh's $K$-theory \cite[3.1]{TT},
i.e., it is a presheaf of spectra $\mrm{K}$ 
such that for every quasi-compact quasi-separated scheme $X$, 
$\mrm{K}(X)$ is the $K$-theory spectrum of the Waldhausen category of perfect complexes on $X$.

\begin{definition}
\leavevmode
\begin{enumerate}[label=\upshape{(\roman*)}]
\item We define a presheaf $\mrm{K}^\rel$ of spectra on $\MSm$ by
\[
	\mrm{K}^\rel((X,D)) = \mrm{K}(X,D) = \hofib(\mrm{K}(X)\to \mrm{K}(D)).
\]
\item We define a presheaf $\bb{Z}^\rel$ on $\MSm$ by
\[
	\bb{Z}^\rel((X,D)) = \begin{cases} \bb{Z} &\text{if }D=\emptyset \\ 0 &\text{if }D\ne\emptyset. \end{cases}
\]
\end{enumerate}
\end{definition}

We use Bousfield-Kan's $\bb{Z}$-completion in \cite{BK72} as a functorial model of Quillen's plus construction.
The $\bb{Z}$-completion is an endofunctor $\bb{Z}_\infty\colon \mcal{S}\to \mcal{S}$ of the category of spaces (= simplicial sets) with a natural transformation $\mrm{Id}_{\mcal{S}}\to \bb{Z}_\infty$.
We will apply $\bb{Z}_\infty$ sectionwise to simplicial presheaves.

The following is a relative and functorial version of Quillen's ``$+=Q$'' theorem.
\begin{theorem}\label{thm:Krel}
There exists an isomorphism
\[
	\Omega^\infty \mrm{K}^\rel \simeq \bb{Z}^\rel \times \bb{Z}_\infty \mbf{X}^\rel
\]
in $\Ho(s\PSh_*(\MSm))$.\footnote{In fact, the isomorphism exists in the Zariski-local homotopy category of simplicial presheaves over any reasonable category of pairs of schemes.}
Under this isomorphism, the multiplication of $\Omega^\infty \mrm{K}^\rel$ coming from loop composition
is compatible with the one of $\bb{Z}^\rel \times \bb{Z}_\infty \mbf{X}^\rel$ coming from the group law of $\bb{Z}$ and the diagonal sum of matrices.
\end{theorem}
\begin{proof}
As in \cite[11.3.6]{Loday}, for any ring $A$ with an ideal $I$, there exists an isomorphism
\[
	\Omega^\infty \mrm{K}(A,I) \simeq \mrm{K}_0(A,I)\times \bb{Z}_\infty\mbf{X}(A,I) .
\]
The construction of the isomorphism can be functorial for the connected components (cf.\ \cite[Proposition 2.15]{Gi81} for the case $I=0$), and thus the desired isomorphism follows.
We can verify easily that each step of the construction of the isomorphism is compatible with the multiplications.
\end{proof}

\begin{theorem}\label{thm:stability}
For $r\ge 2l+2$, the canonical map
\[
	\bb{Z}_\infty\mbf{X}^\rel_r \to \bb{Z}_\infty\mbf{X}^\rel
\]
is a Zariski-local $l$-equivalence of simplicial presheaves on $\MSm$,
i.e., a Zariski-local weak equivalence after taking the $l$-th Postnikov filtration.
\end{theorem}
\begin{proof}
By Suslin's stability as formulated in \cite[\S5]{Be}, for any local ring $A$ with an ideal $I$, the canonical map $\mbf{X}_r(A,I) \to \mbf{X}(A,I)$ induces homology isomorphisms in degree less or equal to $(r-1)/2$.
Then it follows from \cite[Ch I 6.2]{BK72} that the morphism
\[
	\pi_l\bb{Z}_\infty\mbf{X}_r(A,I) \to \pi_l\bb{Z}_\infty\mbf{X}(A,I)
\]
is an isomorphism for $l\le (r-2)/2$.
This proves the theorem.
\end{proof}

In Definition \ref{def:c_Vol}, we have constructed maps
\[
	c_i\colon \mbf{X}_r^\rel \to K(z^i_\rel,0)
\]
for $1\le i\le r$ in $\Ho(s\PSh_*(\MSm^*))$.
Let $l>0$.
For $r\gg l$, we have the following sequence of morphisms in $\Ho(s\PSh_*(\MSm^*))$:
\[
\xymatrix@R-0.5pc{
	\Omega^\infty \mrm{K}^\rel \ar@{=}[r]^-\sim_-{\ref{thm:Krel}} \ar@/^3pc/@{.>}[rrd]^(0.8){\tau_{\le l}\msf{C}_i}
		& \bb{Z}^\rel \times \bb{Z}_\infty \mbf{X}^\rel \ar[d]^{\text{projection}} & \\
		& \bb{Z}_\infty \mbf{X}^\rel \ar[d]^{\text{canonical}} & K(\tau_{\le l}z^i_\rel,0) \\
		& P_l \bb{Z}_\infty \mbf{X}^\rel & P_l K(z^i_\rel,0) \ar[d]_\simeq \ar@{=}[u]^\wr \\
		& P_l \bb{Z}_\infty \mbf{X}_r^\rel \ar[u]_\simeq^{\ref{thm:stability}} \ar[r]^-{P_l c_i}  & P_l\bb{Z}_\infty K(z^i_\rel,0) 
} 
\]
where $P_l$ is the $l$-th Postnikov filtration and $\tau_{\le l}$ is the $l$-th canonical filtration.
According to Lemma \ref{CorStabMap}, the composite $\tau_{\le l}\msf{C}_i$ is independent of the choice of $r$.
Also, the diagram
\[
\xymatrix@C+1pc{
	\Omega^\infty \mrm{K}^\rel \ar[r]^-{\tau_{\le l+1}\msf{C}_i} \ar[rd]_-{\tau_{\le l}\msf{C}_i} & K(\tau_{\le l+1}z^i_\rel,0) \ar[d] \\
	& K(\tau_{\le l}z^i_\rel,0)
}
\]
commutes, where the vertical map is the obvious one.

\begin{theorem}\label{thm:ChernModulus}
Let $i>0$.
There exists a morphism
\[
	\msf{C}_i\colon \Omega^\infty \mrm{K}^\rel \to ``\lim_l"K(\tau_{\le l}z^i_\rel,0)
\]
in $\pro\Ho(s\PSh_*(\MSm^*))$.
For $n\ge 0$, its $(-n)$-th hypercohomology on a modulus pair $(X,D)$ yields a map
\[
	\msf{C}_{n,i}\colon \mrm{K}_n(X,D) \to H^{-n}_\Nis((X,D),z^i_\rel),
\]
which is functorial in $(X,D)\in\MSm$ and is a group homomorphism for $n>0$.
This map coincides with Bloch's Chern class \cite[\S7]{Bl86} when $D=\emptyset$.
\end{theorem}
\begin{proof}
We define $\msf{C}_i = ``\lim_l"\tau_{\le l}\msf{C}_i$.
Since the Nisnevich cohomological dimension of $X$ is finite, by taking the $(-n)$-th hypercohomology of $\msf{C}_i$, we obtain
\begin{multline*}
	\msf{C}_{n,i} \colon \mrm{K}_n(X,D) \xrightarrow{\simeq} H^{-n}_\Nis((X,D),\mrm{K}^\rel) \\ 
		\to H^{-n}_\Nis((X,D),\tau_{\le l}z^i_\rel) \simeq H^{-n}_\Nis((X,D),z^i_\rel).
\end{multline*}
The first map is an isomorphism by Thomason-Trobaugh's Nisnivich descent.
Recall that $\MSm^*$ could be the category over any finite diagram in $\MSm$, which ensures the functoriality.
The map $\msf{C}_{n,i}$ is a group homomorphism for $n>0$ since it is defined by taking the $n$-th homotopy groups.
Compatibility with Bloch's Chern class is immediate from the construction.
\end{proof}

%% file: WhitneyNMJv2.tex
\section{Whitney sum formula}\label{Sec:Whitney}

We show the Whitney sum formula for $\msf{C}_{0,*}$
by doing some more cycle computation.
It involves the operation called {\em algebraic join}.

\subsection{Algebraic join}\label{SecJoin}
Let $X$ be a scheme.
Consider the projective spaces over $X$:
$\bbP ^{r-1}_X=\mbf{Proj} \bigl( \mcal{O}_X[T_{1},\dots ,T_{r}] \bigr) $,
$\bbP ^{s-1}_X=\mbf{Proj} \bigl( \mcal{O}_X[T_{r+1},\dots ,T_{r+s}] \bigr) $
and
\begin{equation}\label{CoordConv}  
\bbP ^{r+s-1}_X=\mbf{Proj} \bigl( \mcal{O}_X[T_{1},\dots ,T_{r+s}] \bigr) .\end{equation}
The schemes $\bbP ^{r-1}_X$ and $\bbP ^{s-1}_X$ are naturally closed subschemes of $\bbP ^{r+s-1}_X$.
We consider the rational maps
$q_1\colon \bbP ^{r+s-1}_X\dashrightarrow \bbP ^{r-1}_X$ and 
$q_2\colon \bbP ^{r+s-1}_X \dashrightarrow \bbP ^{s-1}_X$
defined by $(T_{1},\dots ,T_{r+s})\mapsto (T_{1},\dots ,T_{r})$ and $\mapsto (T_{r+1},\dots ,T_{r+s})$.

Denote by $\pi _1\colon P_1\to \bbP ^{r+s-1}_X$ the blow-up along the ill-defined locus $\bbP ^{s-1}_X$ of $q_1$. Then $q_1$ induces a morphism $q_1'\colon P_1\to \bbP ^{r-1}_X$ which is a $\bbP ^s$-bundle.
Denote by $q_1^*$ the operation on cycles defined as flat pull-back $q_1^{\prime *}$ followed by proper push-forward $\pi _{1*}$.
Similarly, if $\pi _2\colon P_2\to \bbP ^{s-1}_X$ is the blow-up along $\bbP ^{r-1}_X$, the rational map $q_2$ induces a morphism $q_2'\colon P_2\to \bbP ^{s-1}_X$ which is a $\bbP ^r$-bundle.
Denote by $q_2^*$ the flat pull-back $q_2^{\prime *}$ followed by push-forward $\pi _{2*}$.

Observe the obvious fact that
the cycle in $\bbP ^{r-1}_X$ given by a set of homogeneous equations $\{ f_\alpha (T_{1},\dots ,T_{r}) \} _\alpha $ is mapped by $q_1^*$ to the cycle
in $\bbP ^{r+s-1}_X$ defined by the same equations.

\begin{lemma}\label{JoinProduct}
Let $\alpha $ be an element in $z^i(\bbP ^{r-1}_X,m)$ and $\beta $ be an element in $z^j(\bbP ^{r-1}_X,n)$. Suppose that the intersection product $\alpha \cupp \beta \in z^{i+j}(\bbP ^{r-1}_X,m+n)$ is defined. Then the same holds for cycles $q_1^*\alpha \in z^i(\bbP ^{r+s-1}_X,m)$ and $q_1^*\beta \in z^j(\bbP ^{r+s-1}_X,n)$ and we have an equality in $z^{i+j}(\bbP ^{r+s-1}_X,m+n)$:
\[  q_1^*(\alpha \cupp \beta )=(q_1^*\alpha )\cupp (q_1^*\beta ).    \]
The same is true for the operation $q_2^*$.
\end{lemma}
\begin{proof}
Preservation of intersection product certainly holds for flat pull-back.
It holds for proper push-forward by birational maps $\pi $ when the two cycles $\alpha ',\beta '$ under consideration satisfy:
\begin{itemize}
\item The intersection product $(\pi _*\alpha ')\cupp (\pi _*\beta ')$ is again defined, and 
\item no component of $\alpha ',\beta ', \alpha '\cupp \beta '$ or $(\pi _*\alpha ') \cupp (\pi _*\beta ')$ is contained in the exceptional locus of $\pi $.
\end{itemize}
This condition is satisfied in our case.
\end{proof}

\begin{definition}
Let $\alpha \in z^i(\bbP ^{r-1}_X,m)$ and $\beta \in z^j(\bbP ^{s-1}_X,n)$ be cycles. Consider the cycles $q_1^*\alpha \in z^i(\bbP ^{r+s-1}_X,m)$ and $q_2^*\beta \in z^j(\bbP ^{r+s-1}_X,n)$. When the intersection $(q_1^*\alpha )\cupp (q_2^*\beta )\in z^{i+j}(\bbP ^{r+s-1}_X,m+n)$ is well-defined, we denote it by $\alpha \# \beta $.
\end{definition}
The operation $(\alpha ,\beta )\mapsto \alpha \# \beta $ is called {\em algebraic join}. It has been systematically used by authors like Friedlander, Lawson, Michelsohn and Walker. The authors of the present article learned this technique mainly in \cite{FL98}.

\subsection{The equalities}

Let $r,s$ be non-negative integers.
We keep the coordinate convention \eqref{CoordConv}
when considering projective bundles $\bbP (E\GL _r)$ and $\bbP (E\GL _s)$.
On the category $\MSm ^*/\Volo _r^\rel \times \Volo _s^\rel $ we consider presheaves:
\begin{itemize}
\item
$p_{1*}z^i_\rel $, which is induced from $p_*z^i_\rel $ on $\MSm ^*/\Volo _r^\rel $ by the first projection $\Volo _r^\rel \times \Volo _s^\rel \to \Volo _r^\rel $;

\item
$p_{2*}z^i_\rel $, induced by the second projection
$\Volo _r^\rel \times \Volo _s^\rel \to \Volo _s^\rel $;

\item
$p_*z^i_\rel $, induced by the inclusion $\Volo _r^\rel \times \Volo _s^\rel \hookrightarrow \Volo _{r+s}^\rel $;

\item
their non-modulus counterparts $p_{1*}z^i$, $p_{2*}z^i$ and $p_*z^i$
\end{itemize}
Let us distinguish the pull-back maps by
writing $p^{*}_1\colon z^i_\rel \to p_{1*}z^i_\rel $, 
$p_2^{*}\colon z^i_\rel \to p_{2*}z^i_\rel $
and
$p^{*}\colon z^i_\rel \to p_*z^i_\rel $.
Similarly, we can consider three different versions of $\mcal{Z}$'s, denoted by $\mcal{Z}_1$, $\mcal{Z}_2$ and $\mcal{Z}$. There are obvious maps $\mcal{Z}\to \mcal{Z}_1$ and $\mcal{Z}\to \mcal{Z}_2$.

We may consider the partially defined join operator
$\# \colon p_{1*}z^i_\rel 
\otimes p_{2*}z^j_\rel 
\dashrightarrow 
p_*z^{i+j} _\rel $
and its non-modulus version.
The modulus version is a well-defined map in $D(\MSm ^*/\Volo ^\rel _r\times \Volo ^\rel _s)$ (\S \ref{Subsec:JoinWellDef}).
Given two maps $\alpha \colon \mcal{Z}\to p_{1*}z^i$ and $\beta \colon \mcal{Z}\to p_{2*}z^j$,
we consider their cup product followed by algebraic join to get a map
$\alpha \# \beta \colon \mcal{Z}\to p_* z^{i+j} $ 
whenever it is well-defined.

In Remark \ref{AltDef} (ii), we defined maps 
$C^i_{\MSm ^*/\Volo _r^\rel }$
of presheaves on $\MSm ^*/\Volo _r^\rel $.
Via $\mcal{Z}\to \mcal{Z}_1$,
it induces a map 
\begin{equation*}
C^i_{\MSm ^*/\Volo _r^\rel }\colon \mcal{Z}\to p_{1*}z^i_{k(\GL _r)} 
\quad \text{ on }\MSm ^*/\Volo _r^\rel \times \Volo _s^\rel 
 \end{equation*}
which we denote by the same symbol.
Similarly for $C^j_{\MSm ^*/\Volo _s^\rel }$. 
Recall that they depend on the choice of $i$ indices out of $\{ 1,\dots ,r \} $ and $j$ out of $\{ r+1,\dots ,r+s \}$ although it is not explicit in the notation.


Also, thanks to Remark \ref{AltDef} (i)(ii),
we have yet other maps 
\begin{equation*}
C^{k}_{\MSm ^*/\Volo _r^\rel \times \Volo _s^\rel }\colon \mcal{Z}\to p_*z^k_{\rel ,k(\GL _r\times \GL _s)} 
 \end{equation*}
for integers $1\le k \le r+s $.

\begin{proposition}\label{WhitneyCycle1}
For any $0\le i\le r$ and $0\le j\le s$, the two maps
\[ C^i_{\MSm ^*/\Volo _r^\rel }\# 
C^j_{\MSm ^*/\Volo _s^\rel } \text{\upshape \ and } 
C^{i+j}_{\MSm ^*/\Volo _{r+s}^\rel }
\colon \quad
\mcal{Z}\rightrightarrows p_*z^{i+j}_{k(\GL _r\times \GL _s)}  \]
are equal as maps of presheaves on $\MSm ^*/\Volo _r^\rel \times \Volo _s^\rel $
under appropriate choices of indices (specified in the proof). 
In particular we have $\xi ^r_\rel \# \xi ^s_\rel =\xi ^{r+s}_\rel $ 
as maps $\bbZ \rightrightarrows p_*z^{r+s}_\rel $
in $D(\MSm ^*/\Volo _r^\rel \times \Volo _s^\rel )$.
\end{proposition}
\begin{proof}
It suffices to prove the corresponding equality of cycles on the simplicial scheme $\bbP (E\GL _{r+s})_{|B\GL _r\times B\GL _s}$.
By the definition of algebraic join of maps,
the problem is to prove the equality of maps:
\begin{multline*} 
q_1^*(\Gamma \super{1}_{B\GL _r}\cupp \dots \cupp \Gamma \super{i}_{B\GL _r})\ \cupp \ q_2^*(\Gamma \super{r+1}_{B\GL _s}\cupp \dots \cupp \Gamma \super{r+j}_{B\GL _s})
\\%
=\Gamma \super{1}_{B\GL _{r+s}}\cupp \dots \cupp
\Gamma \super{i}_{B\GL _{r+s}}\cupp 
\Gamma \super{r+1}_{B\GL _{r+s}}
\cupp \Gamma \super{r+j}_{B\GL _{r+s}}   
\end{multline*}
from the cone of
\begin{equation*}
\bbZ [B_\bullet \GL _r^*\times B_\bullet \GL _s^*]\to \bbZ [B_\bullet (\GL _{r}\times \GL _s )]\oplus  \bbZ [B_\bullet \GL _r^*\times B_\bullet \GL _s^*]
 \end{equation*}
to 
$ z^{r+s}( \bbP ^{r+s-1}_{k(\GL _r\times \GL _s)} \times - ,\bullet )$.
By Lemma \ref{JoinProduct}, it is reduced to the equalities of cycles
$q_1^*\Gamma \supers{a}_{B\GL _r}(S) =\Gamma \supers{a}_{B\GL _{r+s}}(S)
$
on
$\bbP (E_n\GL _{r+s})_{|B\GL _r^*\times B\GL _s^*}
$
for all $1\le a\le r$ and non-empty subsets 
$S\subset [n]$, and its variants involving $\circ $, $\circ *$ and $q_2^*$.
In view of the fact observed before Lemma \ref{JoinProduct}, this last equality clearly holds. This completes the proof of the proposition.
\end{proof}

Now, consider the following diagram in $D(\MSm ^*/\Volo _{r}^\rel \times \Volo _s^{\rel } )$ (see \S\S \ref{Sec:FurtherExample}, \ref{Subsec:JoinWellDef} for the tensor products $\Dtimes $):
\begin{equation}\label{CommutesInVolo1}
\vcenter{ \xymatrix{
\bigl( \bbZ \oplus \displaystyle\bigoplus _{i=1}^r z^i_{\rel } \bigr)\Dtimes
\bigl( \bbZ \oplus  \displaystyle\bigoplus _{j=1}^s z^j_{\rel } \bigr)
\ar[d]_{\text{intersection}}
\ar[rrr]^(0.6){\sigma (r)\otimes \sigma (s)}
&{}&{}&
p_{1*}z^{r}_\rel \Dtimes p_{2*}z^s_\rel 
\ar[d]^{\# }
\\%
 \bbZ \oplus \displaystyle\bigoplus _{k=1}^{r+s}z^k_\rel 
\ar[rrr]_{\sigma (r+s)}
&{}&{}&
p_*z^{r+s}_{\rel }
} }
\end{equation}
where the vertical map 
``intersection'' sends an element $(\alpha _0,(\alpha _i)_i)\otimes (\beta _0,(\beta _j)_j)$ to the tuple of cycles
$\left( \sum\limits _{k=i+j} \alpha _i\cupp \beta _j \right) _{0\le k\le r+s} $.  
The horizontal maps $\sigma $ are defined by
$\sigma (r)\colon (\alpha _0,(\alpha _i)_{i=1}^r)\mapsto  \alpha _0\xi ^r_\rel +\sum\limits _{i=1}^rp^*(\alpha _i)\cupp \xi ^{r-i}. $ 
Applying Lemma \ref{JoinProduct}, Proposition \ref{WhitneyCycle1} 
and the commutativity of intersection product in the derived category,
one checks that 
the diagram \eqref{CommutesInVolo1} commutes.


The rank $r$ Chern classes $c_i$ are characterized by the property that the composite map
$\bbZ \xrightarrow{(1,c_1,\dots ,c_r)}
\bbZ \oplus \bigoplus _{i=1}^rz^i_\rel 
\xrightarrow{\sigma } p_*z^r_\rel $
is zero. 
Of course the same holds for the rank $s$ case.
It follows by the commutativity of \eqref{CommutesInVolo1} that the composite
$\bbZ \xrightarrow{\left( \sum _{i+j=k}c_i\cupp c_j\right) _{k\ge 0}}\bbZ \oplus \bigoplus _{k=1}^{r+s}z^k_\rel \xrightarrow{\sigma } p_*z^{r+s}_\rel  
$
is zero. From the characterization of Chern classes we get:

\begin{proposition}\label{Prop:SecWhitney}
We have an equality $c_k=\sum _{i+j=k}c_i\cupp c_j  
$
(where $c_0:=1$)
of maps $\bbZ \rightrightarrows z^k_\rel $ in $D(\MSm ^*/\Volo _r^\rel \times \Volo _s^\rel )$ for $1\le k\le r+s$.
\end{proposition}

Equivalently, it is an equality of maps
$\Volo _r^\rel \times \Volo _s^\rel \rightrightarrows K(z^k_\rel ,0) $ in the homotopy category of pointed simplicial presheaves $\Ho (s\PSh _* (\MSm ^*))$.

\begin{corollary}\label{cor:SecWhitney}
The diagram in $\Ho(s\PSh_*(\MSm^*))$
\[
\xymatrix@C+6pc{
	\mbf{X}_r^\rel \times \mbf{X}_s^\rel \ar[r]^-{(1,c_1,\dotsc,c_r)\times(1,c_1,\dotsc,c_s)} \ar[d]
		& K\Bigl(\bbZ \oplus \displaystyle\bigoplus_{i=1}^r z^i_{\rel},0\Bigr) \times K\Bigl(\bbZ \oplus \displaystyle\bigoplus_{i=1}^s z^i_{\rel},0\Bigr) \ar[d] \\
	\mbf{X}_{r+s}^\rel \ar[r]^-{(1, c_1,\dotsc,c_{r+s})} & K\Bigl(\bbZ \oplus \displaystyle\bigoplus _{i=1}^{r+s} z^i_{\rel }, 0\Bigr)
}
\] 
is commutative, where the left vertical map is defined by the diagonal sum of matrices and the right one is defined by the intersection product.
\end{corollary}

\subsection{Whitney sum formula on the relative $K$-theory}\label{SubsecWhitney}

We set $\tilde{z}^*_\rel := \bb{Z} \oplus (\bigoplus_{i\ge 1} z^i_\rel)$.
We define the \textit{total Chern class map}
\begin{equation}\label{totC}
	\msf{C}_{\msf{tot}}\colon \Omega^\infty \mrm{K}^\rel\cong \bb{Z}^\rel \times\bb{Z}_\infty \mbf{X}^\rel
		\to \bb{Z}\times ``\lim_l" K(\tau_{\le l}\tilde{z}^*_\rel, 0)
\end{equation}
by the product of the canonical map $\bb{Z}^\rel\to \bb{Z}$ and $(1,\msf{C}_1,\msf{C}_2,\dotsc)$.
We have seen that the diagonal sum of $\mbf{X}^\rel$ and the group law of $\bb{Z}^\rel$ is compatible with the loop composition of $\Omega^\infty \mrm{K}^\rel$ (Theorem \ref{thm:Krel}).
It follows from Corollary \ref{cor:SecWhitney} that the diagram in $\pro\Ho(s\PSh_*(\MSm^*))$
\[
\xymatrix@C+1.5pc{
	\Omega^\infty \mrm{K}^\rel\times\Omega^\infty \mrm{K}^\rel \ar[r]^-{\msf{C}_{\msf{tot}}\times\msf{C}_{\msf{tot}}} \ar[d]
		& \bb{Z}\times\bb{Z}\times ``\displaystyle\lim_l"K(\tau_{\le l}\tilde{z}^*_\rel \otimes \tau_{\le l}\tilde{z}^*_\rel,0)
			\ar[d]^{\text{sum}\times\text{prod}} \\
	\Omega^\infty \mrm{K}^\rel \ar[r]^-{\msf{C}_{\msf{tot}}} & \bb{Z}\times ``\displaystyle\lim_l" K(\tau_{\le l}\tilde{z}^*_\rel, 0)
}
\]
is commutative.
By taking the $0$-th hypercohomology of $\msf{C}_{\msf{tot}}$ on a modulus pair $(X,D)$, we obtain a map
\begin{equation}\label{total0}
	\mrm{K}_0(X,D) \to \bb{Z}\times\{1\}\times \bigoplus_{i\ge 1}H^0_\Nis((X,D),z^i_\rel).
\end{equation}
We regard the target as a group by
\[
	\Bigl(n, 1+\sum_{i\ge 1} \alpha_i\Bigr) \cdot \Bigl(m, 1+\sum_{j\ge 1} \beta_j\Bigr) = \Bigl( n+m, (1+\sum_{i\ge 1}\alpha_i) (1+\sum_{j\ge 1}\beta_j) \Bigr).
\]
It follows from the above commutative diagram that:
\begin{theorem}\label{thm:WhitneyK}
The map (\ref{total0}) is a group homomorphism.
In other words, we have
\[
	\msf{C}_{0,i}(\alpha+\beta) = \sum_{j+k=i,\,  j,k\ge 0} \msf{C}_{0,j}(\alpha)\msf{C}_{0,k}(\beta)
\]
for $\alpha,\beta\in \mrm{K}_0(X,D)$ with the convention $\msf{C}_{0,0}=1$.
\end{theorem}

%% file: ApplicationsNMJv2.tex
\section{Chern character and application}\label{Application}

\subsection{Chern character}\label{Chernch}

We set
\[
	A^0=\Hom(\Omega^\infty \mrm{K}^\rel,\bb{Z}) \quad \text{and} \quad A^i=\Hom(\Omega^\infty \mrm{K}^\rel,``\lim_l"K(\tau_{\le l}z^i_\rel,0)),
\]
where the Hom group is taken in the category $\pro\Ho(s\PSh_*(\MSm^*))$.
Under this convention, the total Chern class (\ref{totC}) is in the set $A^0\times\{1\}\times\prod_{i\ge 1}A^i$.
We define a map
\[
	\mrm{ch} \colon A^0\times \{1\} \times \bigoplus_{i\ge 1}A^i \to A^*_{\bb{Q}}:=\prod_{i\ge 0}A^i\otimes\bb{Q}
\]
as in \cite[Expos\'e 0, Appendix 1.26]{SGA6}, i.e., 
\[
	\mrm{ch}\Bigl(\Bigl(n,1+\sum_{i\ge 1}x^i\Bigr)\Bigr) = n +\eta\Bigl(\log\Bigl(1+\sum_{i\ge 1}x^i\Bigr)\Bigr),
\]
where $\eta$ is an endomorphism of $A^*_{\bb{Q}}$ defined by $\eta(x^i) = (-1)^{i-1}x^i/(i-1)!$.

The image of $\msf{C}_{\msf{tot}}$ by $\mrm{ch}$ gives a map
\begin{equation}\label{ch}
	\msf{ch}\colon \Omega^\infty \mrm{K}^\rel \to ``\lim_l"K(\tau_{\le l}(\tilde{z}^*_\rel)_{\bb{Q}},0)
\end{equation}
in $\pro\msf{Ho}(s\PSh_*(\MSm^*))$.
According to Theorem \ref{thm:ChernModulus} and Theorem \ref{thm:WhitneyK}, we obtain the following result.
\begin{theorem}\label{thm:ch}
Let $(X,D)\in\MSm$ and $n\ge 0$.
The $(-n)$-th hypercohomology of (\ref{ch}) yields a group homomorphism
\[
	\msf{ch}_n\colon \mrm{K}_n(X,D) \to H^{-n}_{\Nis}((X,D),(\tilde{z}^*_\rel)_{\bb{Q}}),
\]
which is functorial in $(X,D)$ and coincides with Bloch's Chern character when $D=\emptyset$.
\end{theorem}

For an additive category $\mcal{A}$, let $\mcal{A}_{\bb{Q}}$ be the \textit{category up to isogeny},
which has the same objects as $\mcal{A}$ and $\Hom_{\mcal{A}_{\bb{Q}}}(M,N)= \Hom_{\mcal{A}}(M,N)\otimes \bb{Q}$.
We denote the image of $M\in\mcal{A}$ in $\mcal{A}_{\bb{Q}}$ by $M_{\bb{Q}}$.

For a presheaf $F$ on $\MSm$, we define a pro system of presheaves $\hat{F}$ by
\[
	\hat{F}((X,D))= \{F(X,mD)\}_{m\ge 1}.
\]
The above argument can be modified to obtain a map
\begin{equation}\label{chhat}
	\hat{\msf{ch}}\colon (\Omega^\infty\hat{\mrm{K}}^\rel)_{\bb{Q}} \to ``\lim_l"K(\tau_{\le l}\hat{\tilde{z}}^*_\rel,0)_{\bb{Q}}
\end{equation}
in $\pro\msf{Ho}(\pro s\PSh_*(\MSm^*))_{\bb{Q}}$.
Here is a variant of Theorem \ref{thm:ch}.
\begin{theorem}\label{thm:chhat}
Let $(X,D)\in\MSm$ and $n\ge 0$.
The $(-n)$-th hypercohomology of (\ref{chhat}) yields a morphism
\[
	\msf{ch}_n\colon \{\mrm{K}_n(X,mD)\}_{m,\bb{Q}} \to \{H^{-n}_{\Nis}((X,mD),\tilde{z}^*_\rel)\}_{m,\bb{Q}}
\]
in the category of pro abelian groups $(\pro\msf{Ab})_{\bb{Q}}$ up to isogeny.
This is functorial in $(X,D)$ and coincides with Bloch's Chern character when $D=\emptyset$.
\end{theorem}

\subsection{Relative motivic cohomology of henselian dvr}

Let $k$ be a field of characteristic zero.
Let $A$ be a henselian dvr over $k$ and $\pi$ its uniformizer.
Set $X=\Spec A$ and $D=\Spec A/\pi$.
In this section, we prove the following.
\begin{theorem}\label{th:henseliandvr}
For every $n\ge 0$, there is a natural isomorphism
\begin{multline*}
	\{\CH^*(X|mD,n)\}_{m,\bb{Q}} \\
		\simeq \{\mrm{K}_n(X,mD)\oplus \ker(\CH^*(X|mD,n) \to \CH^*(X,n)) \}_{m,\bb{Q}}
\end{multline*}
in the category $(\pro\msf{Ab})_{\bb{Q}}$ of pro abelian groups up to isogeny.
\end{theorem}

We expect that $\{\ker(\CH^i(X|mD,n) \to \CH^i(X,n))\}_{m,\bb{Q}}$ vanishes.

\begin{lemma}\label{lem:keyK}
The canonical map
\[
	\mrm{K}_n(A) \to \{\mrm{K}_n(A/\pi^m)\}_m
\]
is a pro epimorphism.
\end{lemma}
\begin{proof}
By Artin's approximation theorem \cite{Ar69}, we may replace $A$ by its completion $\hat{A}\simeq F[[t]]$, i.e., enough to show that
\[
	\mrm{K}_n(F[[t]]) \to \{\mrm{K}_n(F[t]/t^m)\}_m
\]
is a pro epimorphism.

Since $\mrm{K}_n(F[[t]]) \to \mrm{K}_n(F)$ is a split surjection, it suffices to show that 
\[
	\mrm{K}_n(F[[t]],(t)) \to \{\mrm{K}_n(F[t]/t^m,(t))\}_m
\]
is a pro epimorphism.
By Goodwillie's theorem \cite{Go86} and the $\HC$ version of pro HKR-theorem \cite[Theorem 3.23]{Mo15}, we have pro isomorphisms
\[
\begin{split}
	\{\mrm{K}_n(F[t]/t^m,(t))\}_m &\simeq \{\HC_{n-1}(F[t]/t^m,(t))\}_m \\
								  &\simeq \Bigl\{\bigoplus_{p=0}^{n-1} H^{2p-(n-1)}(\Omega_{F[t]/t^m,(t)}^{\le p})\Bigr\}_m,
\end{split}
\]
where $\Omega^j_{A,I}:= \ker(\Omega^j_A \to \Omega^j_{A/I})$.
By the Poincar\'e lemma \cite[Corollary 9.9.3]{Wei94}, we have
\[
	H^j(\Omega_{F[t]/t^m,(t)}^\bullet) = 0
\]
for every $m,j\ge 0$.
Hence, it follows that
\[
	\{\mrm{K}_n(F[t]/t^m,(t))\}_m \simeq \{\Omega^{n-1}_{F[t]/t^m,(t)}/d\Omega^{n-2}_{F[t]/t^m,(t)}\}_m.
\]

Consider the commutative diagram
\[
\xymatrix{
	0 \ar[r] & H^{n-1}(\Omega^\bullet_{F[t]/t^m}) \ar[r] \ar[d]^\simeq
		& \Omega^{n-1}_{F[t]/t^m}/d\Omega^{n-2}_{F[t]/t^m} \ar[r] \ar@{->>}[d] 	& d\Omega^{n-1}_{F[t]/t^m} \ar[r] \ar@{->>}[d] & 0 \\
	0 \ar[r] & H^{n-1}(\Omega^\bullet_F) \ar[r] & \Omega^{n-1}_F/d\Omega^{n-2}_F \ar[r] & d\Omega^{n-1}_F \ar[r] & 0,
}
\]
where the rows are exact and the vertical maps are split surjections.
Again by the Poincar\'e lemma, the left vertical map is an isomorphism, and thus we have an isomorphism
\[
	\Omega^{n-1}_{F[t]/t^m,(t)}/d\Omega^{n-2}_{F[t]/t^m,(t)} \simeq d\Omega^{n-1}_{F[t]/t^m}/d\Omega^{n-1}_F \xleftarrow[d]{\simeq} tF[t]/t^m \otimes_F \Omega^{n-1}_F.
\]

Given an element $f \otimes d\log y_1\wedge\dotsb\wedge d\log y_{n-1} \in tF[t]/t^m \otimes_F \Omega^{n-1}_F$,
the element $\{\exp(f),y_1,\dotsc,y_{n-1}\} \in \mrm{K}^M_n(F[[t]])$ lifts it via 
\[
\xymatrix@1{
	\mrm{K}^M_n(F[[t]])\ar[r]^-{d\log} & \Omega^n_{F[t]/t^m} & \; tF[t]/t^m \otimes_F \Omega^{n-1}_F \ar@{_{(}->}[l]_-d .
}
\]
Therefore, the composite
\[
	\mrm{K}_n(F[[t]],t) \to \{\mrm{K}_n(F[t]/t^m,(t))\}_m \simeq \{tF[t]/t^m \otimes_F \Omega^{n-1}_F\}
\]
is isomorphic to a levelwise epimorphism, and thus the first map is a pro epimorphism.
This proves the lemma.
\end{proof}

\begin{proof}[Proof of Theorem \ref{th:henseliandvr}]

The Chern character $\hat{\ch}_n$ in Theorem \ref{thm:chhat} fits into the commutative diagram
\[
\xymatrix{
	0 \ar[r] & \{\mrm{K}_n(A,(\pi)^m)\}_{m,\bb{Q}} \ar[r] \ar[d]^{\hat{\ch}_n}
		& \mrm{K}_n(A)_{\bb{Q}} \ar[r]^-\beta \ar[d]^{\ch_n}_\simeq & \{\mrm{K}_n(A/\pi^m)\}_{m,\bb{Q}} \ar[r] & 0 \\
	& \{\CH^*(X|mD,n)\}_{m,\bb{Q}} \ar[r]^-\alpha & \CH^*(X,n)_{\bb{Q}} & &
}
\]
in $(\pro\msf{An})_{\bb{Q}}$.
By Lemma \ref{lem:keyK}, the upper sequence is exact.
By Bloch's comparison theorem in \cite{Bl86}, the middle vertical map $\ch_n$ is an isomorphism.
Hence, it follows that the left vertical map $\hat{\ch}_n$ is a pro monomorphism.

We shall show that the composite
\[
	\Theta:=\beta\circ\ch^{-1}\circ\alpha \colon \{\CH^*(X|mD,n)\}_{m,\bb{Q}} \to \{\mrm{K}_n(A/\pi^m)\}_{m,\bb{Q}}
\]
is the zero map.

Binda-Saito \cite{BS} has constructed the cycle map
\[
	\CH^i(X|mD,n) \to H^{2i-n}(\Omega^{\ge i}_{X|mD}),
\]
where $\Omega^j_{X|mD} = \Omega_A^j(\log D)\otimes A\pi^m$.
Note that we have a pro isomorphism $\{\Omega^j_{X|mD}\}_m \simeq \{\Omega_A^j\otimes A\pi^m\}_m$.
Hence, we have a commutative diagram
\[
\xymatrix{
	\{\CH^i(X|mD,n)\}_m \ar[r] \ar[d]	& \CH^i(X,n) \ar[d] & \\
	\{H^{2i-n}(\Omega_A^{\ge i}\otimes A\pi^m)\}_m \ar[r] & H^{2i-n}(\Omega_A^{\ge i}) \ar[r] & \{H^{2i-n}(\Omega_{A/\pi^m}^{\ge i})\}_m
}
\]
and the bottom composite is zero.
Here, the second vertical map is the usual cycle map to the de Rham cohomology, and the composite
\[
\xymatrix@1{
	\mrm{K}_n(A) \ar[r]^-{\ch} & \CH^*(X,n)_{\bb{Q}} \ar[r] & H^{2*-n}(\Omega_A^{\ge *}) & \HN_n(A) \ar[l]_-\simeq
}
\]
coincides with Goodwillie's Chern character by \cite{Wei93}.
Therefore, the composite
\[
\xymatrix{
	\{\CH^*(X|mD,n)\}_{m,\bb{Q}} \ar[r]^-\Theta & \{\mrm{K}_n(A/\pi^m)\}_{m,\bb{Q}} \ar[d]^c & \\
												& \{\HN_n(A/\pi^m)\}_{m,\bb{Q}} \ar[r]^-\simeq & \{H^{2*-n}(\Omega_{A/\pi^m}^{\ge *})\}_{m,\bb{Q}}
}
\]
equals zero, where $c$ is Goodwillie's Chern character and the last isomorphism is by the pro HKR theorem again.
(The pro HKR theorem may not yield a pro isomorphism for $\HN$ in general, but now the relative part $\HN_n(A/\pi^m,(\pi))$ is equal to $\HC_{n-1}(A/\pi^m,(\pi))$ for which we can apply the pro HKR theorem, and we obtain the above pro isomorphism by the five lemma.)

Consider the commutative diagram
\[
\xymatrix{
	\{\CH^*(X|mD,n)\}_{m,\bb{Q}} \ar[r]^-\Theta & \{\mrm{K}_n(A/\pi^m)\}_{m,\bb{Q}} \ar[r]^-\gamma \ar[d]^c & \mrm{K}_n(F)_{m,\bb{Q}} \ar[d]^{c_1} \\
	 & \{H^{2*-n}(\Omega_{A/\pi^m}^{\ge *})\}_{m,\bb{Q}} \ar[r] & H^{2*-n}(\Omega_F^{\ge *})_{m,\bb{Q}}
}
\]
where $c$ and $c_1$ are Goodwillie's Chern characters and $\gamma$ is the canonical map.
We have seen that $c\circ\Theta =0$, and it is clear that $\gamma\circ\Theta =0$.
We claim that the kernel of $c$ and $c_1$ are isomorphic, which implies that $\Theta=0$.
Indeed, the above square fits into the diagram
\[
\xymatrix@C-1pc{
	\{H^{2*-(n-1)}(\Omega_{A/\pi^m}^{\ge *})\} \ar[r] \ar@{->>}[d] & \{\mrm{K}^{\inf}_n(A/\pi^m)\} \ar[r] \ar[d]^\simeq
			& \{\mrm{K}_n(A/\pi^m)\} \ar[r]^-c \ar[d] & \{H^{2*-n}(\Omega_{A/\pi^m}^{\ge *})\} \ar@{->>}[d] \\
	H^{2*-(n-1)}(\Omega_F^{\ge *}) \ar[r] & \mrm{K}^{\inf}_n(F) \ar[r] & \mrm{K}_n(F) \ar[r]^-{c_1} & H^{2*-n}(\Omega_F^{\ge *})
}
\]
with exact rows.
Here, the first vertical map is surjective and the second vertical map is an isomorphism by Goodwillie's theorem \cite{Go86}.
Hence, $\ker c\simeq \ker c_1$ follows.

Consequently, we obtain a morphism 
\[
	\phi\colon \{\CH^*(X|mD,n)\}_{m,\bb{Q}} \to \{\mrm{K}_n(A,(\pi)^m)\}_{m,\bb{Q}}.
\]
It is clear that $\phi\circ\hat{\ch}_n=\id$ and that $\alpha\circ\hat{\ch}_n\circ\phi = \alpha$.
This completes the proof of Theorem \ref{th:henseliandvr}.
\end{proof}


%% file: LocHtpNMJv2.tex
\section{A lemma on local homotopy theory}\label{S:LocHtp}

The goal in this section is to prove Theorem \ref{lem:LocHtp}.
We fix a small site $\mcal{C}$.
We denote by $\PSh(\mcal{C})$ (resp.\ $s\PSh(\mcal{C})$) the category of presheaves (resp.\ simplicial presheaves) on $\mcal{C}$.
We endow $s\PSh(\mcal{C})$ with the local injective model structure, cf.\ \cite[Theorem 5.8]{Jar15}.
Let us begin with a general construction:
\begin{definition}\label{def:site_fibered}
Let $F\colon\Lambda\to\PSh(\mcal{C})$ be a functor with $\Lambda$ being an arbitrary small category.
We define the {\em site $\mcal{C}/F$ fibered over $F$} as follows:
The objects are all pairs $(X,\lambda,\alpha)$ with $X\in \mcal{C}$, $\lambda\in\Lambda$ and $\alpha\in F(\lambda)(X)$.
The morphisms from $(X,\lambda,\alpha)$ to $(Y,\mu,\beta)$ are all commutative diagrams in the category of presheaves on $\mcal{C}$
\[
\xymatrix{
	X \ar[r]^\alpha \ar[d] & F(\lambda) \ar[d]^{F(\theta)} \\
	Y \ar[r]^\beta & F(\mu)
}
\]
for some morphism $\theta$ in $\Lambda$.
The covering families of $(X,\lambda,\alpha)$ are 
\[
\xymatrix{
	\{U_i\} \ar[r] \ar[rd] & X \ar[d]^\alpha \\
	 & F(\lambda)
}
\]
where $\{U_i\}\to X$ is a covering of $\mcal{C}$.
\end{definition}

In this section, we only use the case $\Lambda=*$ or $\Lambda=\Delta^{\mrm{op}}$.
In the latter case, a functor $F\colon \Delta^{\mrm{op}}\to \PSh(\mcal{C})$ is just a simplicial presheaf.
In principle, we denote by $F$ a simplicial preahseaf.

\subsection{Sites fibered over presheaves}

Let $X$ be a presheaf on $\mcal{C}$.
The forgetful functor $q\colon \mcal{C}/X \to \mcal{C}$ induces an adjunction
\begin{equation}\label{adj-q}
	q^*\colon s\PSh(\mcal{C}/X) \rightleftarrows s\PSh(\mcal{C}) \colon q_*.
\end{equation}
Concretely, for $G\in s\PSh(\mcal{C}/X)$ and $U\in\mcal{C}$, the functor $q^*$ is given by
\[
	q^*(G)(U) = \bigsqcup_{\phi\colon U\to X}G(\phi).
\]
Then, as in the proof of \cite[Lemma 5.23]{Jar15}, we see that $q^*$ preserves cofibrations and local weak equivalences.
Therefore:
\begin{lemma}\label{lem:adj-q}
The adjunction (\ref{adj-q}) is a Quillen adjunction with respect to the local injective model structures.
\end{lemma}

\subsection{Sites fibered over simplicial presheaves}\label{S:simplicialfibered}

Let $F$ be a simplicial presheaf on $\mcal{C}$.
The canonical functor $j_n\colon \mcal{C}/F_n \to \mcal{C}/F$ induces an adjunction
\begin{equation}\label{adj-j_n}
	j_n^*\colon s\PSh(\mcal{C}/F_n) \rightleftarrows s\PSh(\mcal{C}/F) \colon j_{n*}.
\end{equation}
For $G\in s\PSh(\mcal{C}/F_n)$ and $(X \xrightarrow{\alpha} F_m)\in \mcal{C}/F$, we have 
\[
	j_n^*(G)(X \xrightarrow{\alpha} F_m) = \bigsqcup_{\theta\colon [n] \to [m]}G(X \xrightarrow{\theta^*\alpha} F_n).
\]
It follows that $j_n^*$ preserves cofibrations and local weak equivalences.
Hence:
\begin{lemma}\label{lem:adj-j_n}
For every $n\ge 0$, the adjunction (\ref{adj-j_n}) is a Quillen adjunction with respect to the local injective model structures.
\end{lemma}

\begin{remark}\label{rem:projective}
The adjunctions (\ref{adj-q}) and (\ref{adj-j_n}) are also Quillen adjunctions with respect to the local projective model structures.
Since projective fibrations are defined levelwise, it is clear that the forgetful functors $q_*$ and $j_{n*}$ preserve projective fibrations and trivial projective fibrations.
\end{remark}

For simplicial presheaves $G,H$ on $\mcal{C}$, let $\bhom(G,H)$ be the \textit{function complex}, i.e., the simplicial set given by
\[
	\bhom(G,H)_n := \Hom_{s\PSh(\mcal{C})}(G\times \Delta^n,H).
\]
Let $j\colon \mcal{C}/F \to \mcal{C}$ be the forgetful functor, which induces
\[
	j_*\colon s\PSh(\mcal{C}) \to s\PSh(\mcal{C}/F).
\]
Here is the main result in this section, which is a generalization of \cite[Proposition 5.29]{Jar15}.
\begin{theorem}\label{lem:LocHtp}
Let $Z$ be an injective fibrant object in $s\PSh(\mcal{C})$ and $W$ an injective fibrant replacement of $j_*Z$ in $s\PSh(\mcal{C}/F)$.
Then we have a weak equivalence
\[
	\bhom(F,Z) \simeq \bhom(*,W).
\]
In particular, for any presheaf $A$ of complexes of abelian groups on $\mcal{C}$, we have an isomorphism
\[
	H^*(F,A):=\Hom_{\Ho(\mcal{C})}(F,K(A,*)) \simeq H^*(\mcal{C}/F,j_*A).
\]
\end{theorem}

\subsection{Preliminaries to the proof}

\subsubsection{Homotopy limits}
Let $I$ be a small category.
Recall that the homotopy limit of a functor $X\colon I\to s\Set$ ($s\Set=$ the category of simplicial sets) is defined by
\[
	\holim_{i\in I} X(i) := \bhom(B(I\downarrow-), X),
\]
where $I\downarrow-$ is the functor $I\to \msf{Cat}$ assigning the comma category $I\downarrow i$ to each $i\in I$.
Note that the final map $B(I\downarrow-) \to *$ in $s\Set^I$ is a sectionwise weak equivalence.
Hence, in case $X$ is an injective fibrant, we have a weak equivalence
\begin{equation}\label{formula1:holim}
	\holim_{i\in I}X(i) \simeq \bhom(*,X) = \lim_{i\in I} X(i).
\end{equation}

\begin{lemma}\label{lem:holim}
Let $Z$ be a sectionwise fibrant object in $s\PSh(\mcal{C}/F)$.
Then there exists a natural weak equivalence
\[
	\holim_{X\in\mcal{C}/F}Z(X) \simeq \holim_{m\in\Delta} \holim_{X\in \mcal{C}/F_m} Z(X).
\]
\end{lemma}
\begin{proof}
We construct a morphism
\begin{equation}\label{eq1:hocolim}
	\Psi\colon \hocolim_{m\in\Delta^{\mrm{op}}}\Bigl(j_m^*B\bigl((\mcal{C}/F_m)^{\mrm{op}}\downarrow -\bigr)\Bigr) \xrightarrow{\simeq} B\bigl((\mcal{C}/F)^{\mrm{op}}\downarrow -\bigr)
\end{equation}
in $s\PSh(\mcal{C}/F)$, and show that it is a sectionwise weak equivalence between projective cofibrant objects in $s\PSh(\mcal{C}/F)$.
Let $X\xrightarrow{\alpha}F_n$ be an object in $\mcal{C}/F$.
Then we have
\[
	j_m^*\Bigl(B\bigl((\mcal{C}/F_m)^{\mrm{op}}\downarrow -\bigr)\Bigr)(X,\alpha) = \bigsqcup_{\theta\colon[m]\to[n]}B\bigl((\mcal{C}/F_m)^{\mrm{op}}\downarrow(X,\theta\alpha)\bigr).
\]
Hence, the sections at $(X\xrightarrow{\alpha}F_n)$ of the left hand side of (\ref{eq1:hocolim}) are equal to the coequalizer of
\begin{multline}\label{eq2:hocolim}
	\bigsqcup_{[l]\xrightarrow{\sigma}[m]\xrightarrow{\theta}[n]}B\bigl((\mcal{C}/F_m)^{\mrm{op}}\downarrow(X,\theta\alpha)\bigr) \times B(\Delta\downarrow l) \\
		\rightrightarrows \bigsqcup_{\theta\colon[m]\to[n]}B\bigl((\mcal{C}/F_m)^{\mrm{op}}\downarrow(X,\theta\alpha)\bigr) \times B(\Delta\downarrow m).
\end{multline}
For each $\theta\colon[m]\to[n]$, we define a functor
\[
	\bigl((\mcal{C}/F_m)^{\mrm{op}}\downarrow(X,\theta\alpha)\bigr)\times(\Delta\downarrow m) \to \bigl((\mcal{C}/F)^{\mrm{op}}\downarrow (X,\alpha)\bigr)
\]
by sending
\[
\xymatrix@C-1pc{
	X \ar[r]^\alpha \ar[d] 	& F_n \ar[r]^\theta 	& F_m \ar[r]^\sigma \ar@{=}[d] 	& F_l \\
	Y \ar[rr]^\beta 		& 					& F_m 						&
}
\quad\text{to}\quad
\xymatrix{
	X \ar[r]^\alpha \ar[d] 	& F_n \ar[d]^{\sigma\theta} \\
	Y \ar[r]^{\sigma\beta} 	& F_l.
}
\]
These functors induce a morphism of simplicial sets
\[
	\bigsqcup_{\theta\colon[m]\to[n]}B\bigl((\mcal{C}/F_m)^{\mrm{op}}\downarrow(X,\theta\alpha)\bigr) \times B(\Delta\downarrow m)
		\to B\bigl((\mcal{C}/F)^{\mrm{op}}\downarrow (X,\alpha)\bigr),
\]
which is functorial in $(X,\alpha)$ and kills the difference of (\ref{eq2:hocolim}).
Hence, it induces the desired morphism $\Psi$.

The coequalizer of (\ref{eq2:hocolim}) is also equal to
\[
	\hocolim_{\theta\colon[m]\to[n]} B\bigl((\mcal{C}/F_m)^{\mrm{op}}\downarrow (X,\theta\alpha)\bigr),
\]
where $\theta$ runs through $\Delta\downarrow n$, and it is contractible.
It follows that the source and the target of $\Psi$ are sectionwise contractible, and thus $\Psi$ is a sectionwise weak equivalence.

According to \cite[Corollary 14.8.8]{Hir03}, diagrams of the form $B(\mcal{E}\downarrow -\bigr)$ are projective cofibrant.
Since the adjunction (\ref{adj-j_n}) is a Quillen adjunction with respect to the projective model structure (Remark \ref{rem:projective}), $j_m^*B((\mcal{C}/F_m)^{\mrm{op}}\downarrow -)$ is projective cofibrant.
Hence, both sides of (\ref{eq1:hocolim}) are projective cofibrant.

It follows that
\[
\begin{split}
	\holim_{\mcal{C}/F}j_*Z &= \bhom\Bigl( B\bigl((\mcal{C}/F)^{\mrm{op}}\downarrow -\bigr), Z\Bigr) \\
							&\simeq \bhom\biggl( \hocolim_{m\in\Delta^{\mrm{op}}}\Bigl(j_m^*B\bigl((\mcal{C}/F_m)^{\mrm{op}}\downarrow -\bigr)\Bigr), Z\biggr) \\
							&\simeq \holim_{m\in \Delta} \bhom\Bigl( j_m^*B\bigl((\mcal{C}/F_m)^{\mrm{op}}\downarrow -\bigr), Z\Bigr) \\
							&\simeq \holim_{m\in \Delta} \bhom\Bigl(B\bigl((\mcal{C}/F_m)^{\mrm{op}}\downarrow -\bigr), j_{m*}Z\Bigr) \\
							&= \holim_{m\in \Delta} \holim_{\mcal{C}/F_m} j_{m*}Z.
\end{split}
\]
The first isomorphism follows from \cite[18.4.5]{Hir03}, the second one follows from \cite[18.1.10]{Hir03} and the third one is the adjunction (\ref{adj-j_n}) of $j^*_m$ and $j_{m*}$.
\end{proof}

\subsubsection{Cosimplicial spaces}
We call a cosimplicial object in $s\Set$ a cosimplicial space, and denote the category of cosimplicial spaces by $cs\Set$. 
Let $A$ be a cosimplicial space and $S$ a simplicial presheaf on a site $\mcal{C}$.
We define a simplicial presheaf $A\otimes S$ to be the coequalizer of
\[
	\bigsqcup_{\theta\colon [m] \to [n]}A^m\times S_n \rightrightarrows \bigsqcup_{[n]}A^n\times S_n.
\]
Let $X$ be another simplicial presheaf on $\mcal{C}$.
We define a cosimplicial space $\underline{\Hom}(S,X)$ by $\underline{\Hom}(S,X)^n_m := \Hom(S_n,X_m)$.
\begin{lemma}\label{lem:adj-cs}
There is a Quillen adjunction
\begin{equation}\label{adj-cs}
	-\otimes S \colon cs\Set \rightleftarrows s\PSh(\mcal{C}) \colon \underline{\Hom}(S,-)
\end{equation} 
with respect to the Bousfield-Kan model structure on $cs\Set$ \cite[X, \S5]{BK72} and the injective model structure on $s\PSh(\mcal{C})$. 
\end{lemma}
\begin{proof}
It is clear that (\ref{adj-cs}) is an adjunction.
We show that $\underline{\Hom}(S,-)$ preserves fibrations and trivial fibrations.

Let $DS_n$ be the coequalizer of
\[
	\bigsqcup_{i<j}S_{n-2} \rightrightarrows \bigsqcup_i S_{n-1},
\]
which is a subpresheaf of $S_n$.
Then, for a simplicial presheaf $X$, $\bhom(DS_n, X)$ is the $(n-1)$-th matching space (\cite[X, \S4.5]{BK72}) of $\underline{\Hom}(S,X)$.
Let $X\to Y$ be an injective fibration (resp.\ injective trivial fibration) of simplicial presheaves.
Since $DS_n\to S_n$ is a cofibration, the induced map
\[
\xymatrix@R-1.5pc{
	\bhom(S_n, Y) \ar@{=}[d] \ar[r] 	& \bhom(S_n, X) \times_{\bhom(DS_n,X)} \bhom(DS_n,Y) \ar@{=}[d] \\
	\underline{\Hom}(S,Y)^n  		& \underline{\Hom}(S, X)^n\times_{M^{n-1}\underline{\Hom}(S,X)}M^{n-1}\underline{\Hom}(S,Y)
}
\]
is a fibration (resp.\ trivial fibration).
This proves the lemma.
\end{proof}

\subsection{Proof of Theorem \ref{lem:LocHtp}}

Now, we are given an injective fibrant object $Z$ in $s\PSh(\mcal{C})$ and an injective fibrant replacement $W$ of $j_*Z$ in $s\PSh(\mcal{C}/F)$.

Firstly, we show that $j_*Z \to W$ is a sectionwise weak equivalence.
By Lemma \ref{lem:adj-j_n}, $j_{n*} \colon s\PSh(\mcal{C}/F) \to s\PSh(\mcal{C}/F_n)$ preserves injective fibrations.
Put $q_n := j\circ j_n\colon \mcal{C}/F_n \to \mcal{C}$.
By Lemma \ref{lem:adj-q}, $q_{n*}\colon s\PSh(\mcal{C}) \to s\PSh(\mcal{C}/F_n)$ also preserves injective fibrations.
Hence, $q_{n*}Z \to j_{n*}W$ is a local weak equivalence between fibrant objects, and thus a sectionwise weak equivalence for every $n$.

We have seen that $j_*Z \to W$ is a sectionwise weak equivalence between sectionwise fibrant objects.
Hence, by \cite[XI, 5.6]{BK72}, we have a weak equivalence
\[
	\holim_{X\in\mcal{C}/F}Z(X) \simeq \holim_{X\in\mcal{C}/F}W(X).
\]
Since $W$ is an injective fibrant object on $\mcal{C}/F$ (locally, and thus for the indiscrete topology), it follows from (\ref{formula1:holim}) that the right hand side of the above is weak equivalent to $\bhom(*,W)$.
Hence, it remains to show that there is a weak equivalence
\begin{equation}\label{formula3:holim}
	\holim_{X\in\mcal{C}/F}Z(X) \simeq \bhom(F,Z).
\end{equation}

Since $F$ is isomorphic to $\Delta\otimes F$, it follows from the adjunction (\ref{adj-cs}) that we have an isomorphism
\[
	\bhom(F,Z) \simeq \bhom(\Delta,\underline{\Hom}(F,Z)).
\]
Now, $\underline{\Hom}(F,Z)$ is the cosimplicial space whose degree $n$ part is $\lim_{X\in \mcal{C}/F_n}Z(X)$.
Moreover, by Lemma \ref{lem:adj-cs}, $\underline{\Hom}(F,Z)$ is a fibrant cosimplicial space.
Therefore, by \cite[XI, 4.4]{BK72},  
\begin{equation}\label{eq1:LocHtp}
	\bhom(F,Z) \simeq \holim_\Delta \lim_{X\in \mcal{C}/F_n}Z(X).
\end{equation}
Since $q_{n*}Z$ is injective fibrant by Lemma \ref{lem:adj-q}, it follows from (\ref{formula1:holim}) that the canonical map
\[
	\lim_{X\in \mcal{C}/F_n}Z(X) \xrightarrow{\simeq} \holim_{X\in \mcal{C}/F_n}Z(X)
\]
is a weak equivalence between fibrant objects.
Therefore,
\begin{equation}\label{eq2:LocHtp}
	\holim_\Delta \lim_{X\in \mcal{C}/F_n}Z(X) \simeq  \holim_\Delta \holim_{X\in \mcal{C}/F_n}Z(X).
\end{equation}
By Lemma \ref{lem:holim} and (\ref{eq1:LocHtp}, \ref{eq2:LocHtp}), we obtain the desired formula (\ref{formula3:holim}).


%% file: CyclePrelimNMJv2.tex
\section{Preliminaries on algebraic cycles}\label{CyclePrel}

\subsection{Moving lemma with modulus}\label{interProd}

Let $(X,D)\in \MSm $. 
By a {\em family of constructible subsets} $\mcal{C}=\{ C_d \} _{d\in \bbZ }$ of $X\setminus D$ we mean a non-decreasing family $C_d\subset C_{d+1}$ such that $\dim (C_d)\le d$ and $C_{\dim (X)}=X$.
Let $z^i_{\mcal{C}}(X|D,\bullet )\subset z^i(X|D,\bullet )$ be the subcomplex of cycles $V\in z^i(X|D,n)$ such that for every $d\in \bbZ $ and map of cubes
$\square ^{m}\to \square ^n$, the following inequality of dimensions holds:
$\dim \bigl( |V|\times _{X\times \square ^n} (C_d\times \square ^m) \bigr) \le (d+m)-i .
$ 
Of course, it suffices to consider face maps $\square ^m\hookrightarrow \square ^n$.
When $\mcal{C}$ is the trivial family $\mcal{C}_{\mathrm{triv}}$ characterized by $C_{\dim (X)-1}=\emptyset $, it is the same as $z^i(X|D,\bullet )$.
Given an equidimensional $k$-scheme $Y$ of finite type, 
we consider the presheaf $z^i_{\mcal{C}}(-\times Y|D\times Y,\bullet )$ on $X_\Nis $ defined by
\begin{equation*}
\bigl( U\xrightarrow{f}X \bigr) \mapsto
z^i_{\{ f^{-1}(C_{d-\dim (Y)})\times Y\} _{d} } (U\times Y|D_U\times Y,\bullet ).
 \end{equation*}
The case $Y=\Spec (k)$ is of primary importance, but we need the $Y=\bbP ^{r-1}$ case as well when we consider projective bundles.
\begin{theorem}{\upshape \cite[Theorem 2]{Kai}}
\label{MovingVariant}
In the notation as above, the inclusion
$z^i_{\mcal{C}}(-\times Y|D\times Y,\bullet )\hookrightarrow z^i(-\times Y|D\times Y,\bullet )$  
is a quasi-isomorphism on $X_\Nis $.
\end{theorem}
Cycle-theoretic operations often require proper intersection conditions for their well-definedness.
If we can find a family $\mcal{C}$ such that the operation is always defined on $z^i_{\mcal{C}} (-\times Y|D\times Y,\bullet )$,
Theorem \ref{MovingVariant} allows us to conclude that the operation is well-defined in the derived category $D(X_\Nis )$.
For example, for a functor $F\colon \Lambda \to \MSm $ as in \S \ref{SetUpProjBund} and an equidimensional $k$-scheme $Y$, 
the method in \cite[p.94]{Levine} (say)
applied to the morphisms in $\Lambda $
gives a canonical family $\mcal{C}(\lambda )$ on $X_\lambda $
such that the association
\begin{equation*}
((X,D),\lambda ,f)\mapsto z^i_{C(\lambda )} (X\times Y|D\times Y,\bullet )
 \end{equation*}
is a presheaf on $\MSm ^*$. This and a similar argument give the presheaves
$z^i_\rel $ and $p_*z^i_\rel $ in \S \ref{SetUpProjBund}.

\subsubsection{Further example: intersection product}\label{Sec:FurtherExample}

Given a cycle $W \in z^j((X\setminus D)\times Y,n)$,
the cycles $V\in z^i(X|D,m )$ such that the intersection product
$(V \times Y)\cdot W  \in z^{i+j}(X\times Y|D\times Y,m+n)$ 
is well-defined form a subcomplex of the form
$z^i_{\mcal{C}}(X|D,\bullet )$.
By Theorem \ref{MovingVariant}, it is isomorphic to $z^i(-|D,\bullet )$ in $D(X_\Nis )$.
If $W$ vanishes by the differential in $z^j((X\setminus D)\times Y,\bullet )$, we get a map of complexes 
${(-\times Y)\cdot W }\colon  z^i(-|D,\bullet )
\to
z^{i+j}(-\times Y|D\times Y,\bullet )$
in $D(X_\Nis )$.

Or, letting $W$ vary in $z^j((-\setminus D)\times Y,\bullet )$, we get a subcomplex 
$z^i(-|D,\bullet )\Dtimes z^j((-\setminus D)\times Y,\bullet )$ of the usual tensor $\otimes $ where the intersection product is well-defined.
By the fact that a columnwise quasi-isomorphism of bicomplexes (suitably bounded) induces a quasi-isomoprhism on the total complexes, we get a diagram in $D(X_\Nis )$:
\begin{equation*}\begin{array}{ccc}
z^i(-|D,\bullet )\Dtimes z^j((-\setminus D)\times Y,\bullet )&
\xrightarrow{(-\times Y)\cdot (-)}&
z^{i+j}(-\times Y|D\times Y,\bullet ).
\\
\downarrow \simeq &{}
\\
z^i(-|D,\bullet )\otimes z^j((-\setminus D)\times Y,\bullet )&{}
\end{array} \end{equation*}
This is used when we construct maps $ p^*(-)\cdot \xi ^j$ in \S \ref{SubsecProjBundle}.
Also, since $z^j(X|D,\bullet )\subset z^j(X\setminus D,\bullet )$, we get intersection product
$z^i_\rel \otimes z^j_\rel \to z^{i+j}_\rel $
in $D(\MSm ^*)$.

\subsubsection{Yet another example: algebraic join}\label{Subsec:JoinWellDef}
In the situation of \S \ref{SecJoin}, 
for each $W\in z^j(\bbP ^{s-1}_{X\setminus D},n)$,
the cycles $V\in z^i(\bbP ^{r-1}_{X}|\bbP ^{r-1}_D ,m)$ such that the join
$V\# W \in z^{i+j}(\bbP ^{r+s-1}_{X}|\bbP ^{r+s-1}_D,m+n )$ is well-defined form a subcomplex of the form
$z^i_{\mcal{C}}(\bbP ^{r-1}_X|\bbP ^{r-1}_D,\bullet )$ for some $\mcal{C}$ on $X$.
One can find such a $\mcal{C}$ by applying \cite[p.94]{Levine} to the fiber dimensions of the projections $W_{|F}\to X$ with $F\subset \square ^n$ being faces.
By varying $W$ as in the previous paragraph, we get a map 
$\#\colon
z^i(\bbP ^{r-1}_{(-)}|\bbP ^{r-1}_D,\bullet )\Dtimes 
z^j(\bbP ^{s-1}_{(-\setminus D)},\bullet )
\to
z^i(\bbP ^{r+s-1}_{(-)}|\bbP ^{r+s-1}_D,\bullet ) 
$ 
in $D(X_\Nis )$. Carrying out this argument on $\MSm ^*/\Volo _r^\rel \times \Volo _r^\rel $
(or more generally on $\MSm ^*/B\GL _r\times B\GL _s$),
we get the following diagram:
\begin{equation*}
p_{1*}z^i_\rel \otimes p_{2*}z^j
\xleftarrow{\simeq }
p_{1*}z^i_\rel \Dtimes p_{2*}z^j\xrightarrow{\# } 
p_*z^{i+j}_\rel 
 \end{equation*}
which gives the algebraic join in $D(\MSm ^*/B\GL _r\times B\GL _s)$.

\subsection{Computing cup product}\label{Appendix:cup}

Here we give a formula which gives us explicit representatives for the cup product.
This is used in \S\S \ref{GlobalProjBund} and \ref{UnivChern}, where algebraic cycles are involved.
In \S \ref{Sec:PropInt}, we prove results saying that the proper intersection of the lowest-degree representatives (in some sense) implies the proper intersection of all representatives, whereby we get well-defined intersection products of cohomology classes.

Consider a site $\mcal{C}$.
Let us agree that the cup product
of two cohomology classes
$\phi  \in H^i(\mcal{C},{F})$ and $\psi  \in H^j(\mcal{C},G)$
(where ${F},{G}$ are objects in the derived category of complexes of abelian sheaves) is defined as the derived tensor of the two maps
$\bbZ \to F[i]$, $\bbZ \to G[j]$ representing them:
\begin{equation*}
\phi \cdot \psi := [ \bbZ  =\bbZ \otimes ^L\bbZ \xrightarrow{\phi \otimes ^L\psi } F\otimes ^L G[i+j]] \quad \in
H^{i+j}(\mcal{C},F\otimes ^L G).
 \end{equation*}
If we are given a map into another object
$ F\otimes ^L G\to E $,
then we get its image in $H^{i+j}(\mcal{C},E)$ which is often denoted by $\phi \cdot \psi $ again.

If we are given a quasi-isomorphism $\varepsilon\colon \mcal{Z}\to \bb{Z} $ and a morphism $D\colon \mcal{Z}\to \mcal{Z}\otimes ^L \mcal{Z} $ 
such that $(\varepsilon \otimes ^L \varepsilon )\circ D=\varepsilon $ as maps $\mcal{Z}\rightrightarrows \bbZ \otimes ^L \bbZ =\bbZ $,
then the cup product of classes represented by maps
$ \phi  \colon \mcal{Z}\to F[i] $ and $ \psi  \colon \mcal{Z} \to G[j]  $
can be computed as the composition
$ \mcal{Z}\xrightarrow{D} \mcal{Z} \otimes ^L \mcal{Z} \xrightarrow{\phi  \otimes ^L \psi  } F\otimes ^L G[i+j] .$  

\subsubsection{The case of a site fibered over a simplicial presheaf}

Let $\Volo $ be a simplicial presheaf on $\mcal{C}$.
We are interested in the site $\mcal{C}/\Volo $.
Denote by $\Delta $ the simplicial presheaf defined by $(X,n,\alpha ) \mapsto \Delta ^n $.
The projection $\Delta \to \mathrm{pt}$ induces a quasi-isomorphism $\bbZ \otimes \Delta \to \bbZ $. 

For integers $0\le k\le l\le n$, we denote by $[k,l]\subset [n]$ the subset $\{ k,k+1\dots ,l \} $. 
By abuse of notation, let the same symbol also denote the inclusion map $[l-k]\hookrightarrow [n]$ onto it.
The complex $\bbZ \otimes \Delta $ has the coalgebra structure (the Alexander-Whitney map)
$D\colon \bbZ \otimes \Delta \to (\bbZ \otimes \Delta )\otimes (\bbZ \otimes \Delta )$ 
by which $\theta \in \Delta ^n_m$ is mapped to the sum:
\[  \sum _{\begin{subarray}{c}p,q\ge 0\\ p+q=m\end{subarray}} (\theta \circ [0,p])
\otimes 
(\theta \circ [p,p+q])
\in 
\bigoplus _{\begin{subarray}{c}p,q\ge 0\\ p+q=m\end{subarray}} (\bbZ \otimes \Delta ^n_p)\otimes (\bbZ \otimes \Delta ^n_q) .  \]

Now suppose that $\mcal{C}$ is a category of schemes equipped with the Zariski topology (or finer), and that $\Delta $ is covered by two open simplicial subpresheaves $\Delta ^\circ $ and $\Delta ^*$. 
Set $\Delta \os := \Delta ^\circ \cap \Delta ^* $.
In this case we have a weak equivalence
\[ \mcal{Z}:=\cone \left( \bbZ \otimes \Delta \os \xrightarrow[(\mrm{incl.},\mrm{incl.} )]{} (\bbZ \otimes \Delta ^\circ ) \oplus (\bbZ \otimes \Delta ^* )\right) \xrightarrow[\mrm{incl.}\oplus (-\mrm{incl.})]{\sim } \bbZ \otimes \Delta . \]
The complex $\mcal{Z}$ has a coalgebra structure $D\colon \mcal{Z}\to \mcal{Z}\otimes \mcal{Z}$. Writing it down is equivalent to writing down the formula for the cup product, so we do the latter.
Let $\phi \colon \mcal{Z}\to F$ and $\psi \colon \mcal{Z}\to G$ be maps of complexes.
For each object $(X,n,\alpha )\in \mcal{C}/\Volo $ and degree $m$, the map $\phi $ gives the data:
\begin{eqnarray*} \theta \in \Delta ^{\circ }_m(X,\alpha ) 
&\mapsto &\phi ^\circ (X,\alpha ,\theta )\in F(X,\alpha )_m  \\%
\theta \in \Delta ^{* }_m(X,\alpha ) 
&\mapsto &\phi ^*(X,\alpha ,\theta )\in F(X,\alpha )_m  \\
\theta \in \Delta \os _m(X,\alpha ) 
&\mapsto &\phi \os (X,\alpha ,\theta )\in F(X,\alpha )_{m+1}  \end{eqnarray*}
(and similarly $\psi ^\circ (X,\alpha ,\theta )$, $\psi ^* (X,\alpha ,\theta ) $ and $\psi \os (X,\alpha ,\theta )$).
Then their cup product in $H^0(\mcal{C},F\otimes G)$ is represented by the data: 
\begin{eqnarray*} (\phi \cupp \psi )^\circ (X,\alpha ,\theta )
&=& \sum _{p+q=m} 
\phi ^\circ (X,\alpha ,\theta \circ [0,p] )\otimes \psi ^\circ (X,\alpha ,\theta \circ [p,p+q])  \\%
(\phi \cupp \psi )^* (X,\alpha ,\theta )
&=& \sum _{p+q=m} 
\phi ^* (X,\alpha ,\theta \circ [0,p] )\otimes \psi ^* (X,\alpha ,\theta \circ [p,p+q]) \end{eqnarray*}
\begin{multline*}
(\phi \cupp \psi )\os (X,\alpha ,\theta ) 
= \sum _{p+q=m}\biggl\{ (-1)^p \phi ^\circ (X,\alpha ,\theta \circ [0,p])\otimes \psi \os (X,\alpha ,\theta \circ [p,p+q]) \biggr.
\\%
\biggl. +
\phi \os (X,\alpha ,\theta \circ [0,p])\otimes \psi ^* (X,\alpha ,\theta \circ [p,p+q])  \biggr\} .
\end{multline*}

\subsection{Proper intersection lemmas}\label{Sec:PropInt}

The statement of the next lemma may appear to be a little involved, but its proof is easy. (The interested reader can try the $n=0$ case first.)
Below, we deduce some of its consequences which are useful in checking the well-definedness of cup products.

\begin{lemma}\label{Gamma2}
Let $X$ be an algebraic scheme and $V$ a closed subscheme of $X\times \square ^n$.
Let $G,H$ be functions on $X\times \square ^n$ which and whose restrictions to $V$ are nowhere
zero-divisors.
Assume further that $V$, $\Div (G)$, $\Div (H)$,
$V\cap \Div (G)$ and $V\cap \Div (H)$ satisfy the face condition in $X\times \square ^n$.
Then the function $H+t_{n+1}(G-H)$ on $X\times \square ^{n+1}$ and its restriction to $V\times \square ^1$ are nowhere zero-divisors, and the intersection $(V\times \square ^1)\cap \Div (H+t_{n+1}(G-H))$
satisfies the face condition.
\end{lemma}

\subsubsection{Semi-simplicial schemes}\label{Semisimplicial}

In \S\S \ref{GlobalProjBund} and \ref{UnivChern} we are interested in the following situation.
Let $X_\bullet $ be a semi-simplicial scheme with flat face maps and $i\ge 1$ an integer. Let $L\super{a}$ be a line bundle on $X_\bullet $ given for each $a \in \{ 1,\dots ,i \} $ equipped with a section $\sigma \super{a}\in \Gamma (X_0,L\super{a}_0)$ which is everywhere a non zero-divisor.
Section \ref{general-procedure} gives meromorphic functions $F\super{a}_n:=F\super{\sigma \super{a}}_n$ on $X_n\times \square ^n$ and cycles
\[ \Gamma \super{a}_n:= \Div (F \super{L\super{a} ,\sigma \super{a} }_n)\in z^1(X_n,n) . \]
Given a subset $S\subset [n]$ with $s+1$ elements,
we denote again by $S$ the inclusion $[s]\hookrightarrow [n]$ onto $S$.
Write $F\super{a}(S )$ and $\Gamma \super{a}(S )$ for the pull-backs of $F\super{a}_s$ and $\Gamma _s\super{a}$ by the map
$X(S )\times \id _{\square ^s}$
: $X_n\times \square ^s\to X_s\times \square ^s $.
If we denote by $d _k \colon [s-1]\hookrightarrow [s]$ the face maps ($0\le k\le s$), the functions $F\super{a}(S )$ admit an inductive definition (on the size of $S$):
\begin{equation}\label{IotaInductive} 
	F\super{a}(S )=t_s\cdot (S \circ v_s^{[s]})_*\sigma \super{a} +(1-t_s)\left( 
	F\super{a}(S \circ d _s )(t_1,\dots ,t_{s-1})
\right) . \end{equation}

Lemma \ref{Lem:well-def2} in the main body of the article is a consequence of the following.

\begin{lemma}\label{Lem:ProperInt}
Keep the notation above and
let $m\ge 0$ be an integer.
Suppose that the Cartier divisors 
$\Gamma \super{a}(v_{k_a}^{[m]} )$ 
($a=1,\dots ,i$)  
on $X_m$
form a local complete intersection for every choice of indices
$0\le k_1\le \dots \le k_i\le m$.
Then for every choice of non-empty subsets
$S_1\le \dots \le S_i$ of $[m]$, the Cartier divisors on $X_m\times \square ^m$:
\[ \mrm{pr}_{S_a}^*\Gamma \super{a}(S_a)\hspace{30pt} a=1,\dots ,i \]
form a complete intersection, and the intersection satisfies the face condition.
Consequently, their intersection product belongs to $z^i(X_m,m)$.
\end{lemma}
\begin{proof}
If all $\# S_a -1$ are zero, the assertion is the same as the assumption.
The general case follows from the inductive formula \eqref{IotaInductive} and Lemma \ref{Gamma2}. This completes the proof.
\end{proof}

\subsubsection{Variant}\label{TwoOpenSets}
In \S \ref{UnivChern}, we are interested in a little more involved situation where each $X_n $ admits an open cover 
$X_n=X_n^\circ \cup X_n^* $ 
and the collections of schemes $(X_n^\circ )_n$, $(X_n^*)_n$ form semi-simplicial subschemes. Write $X_n\os := X_n^\circ \cap X_n^* $.
Suppose that we are given sections $\sigma \supero{a}\in \Gamma (X_0^\circ ,L_0)$ and $\sigma \supers{a}\in \Gamma (X_0^* ,L_0)$ ($1\le a\le i$) which are everywhere non zero-divisors.
Invariants associated with $\sigma \supero{a}$ are indicated by superscripts $(-)\supero{a}$,
and $\sigma \supers{a}$ by $(-)\supers{a}$.
The homotopy in equation \eqref{EqHomotopy} gives
\begin{equation*}
F\superos{a}_n:= F\super{\sigma \supero{a},\sigma \supers{a}}_n 
\quad\text{and}\quad 
\Gamma _n\superos{a}:=\Gamma \super{\sigma \supero{a},\sigma \supers{a}}_n.
\end{equation*}
For a subset $S\subset [m]$ with $s+1$ elements, let $F\superos{a}(S ) $ be the pull-back of $F\superos{a}_s $ by the map $X(S )\times \id  _{\square ^{s+1}} $ :
$X_m\times \square ^{s+1}\to X_s\times \square ^{s+1}$.

The proof of the following lemma is similar to the previous one.
One applies it to the simplicial schemes $X\times \Delta ^n$
with open covering in degree $m$:
$X\times \Delta ^n_m =\bigsqcup _{\theta \in \Delta ^n_m}\left( \left( X\setminus D \right) \cup \left( X_{\alpha ,\theta} ^* \right) \right)$
to verify the well-definedness of the cup product in formula \eqref{NonModXi^r}, \S \ref{TheComplZ}.

In the lemma, for a subset $S\subset [n]$
we denote by $\mrm{pr}_S$ the projection $\square ^n\to \square ^s $
to the $S(1),\dots ,S(s)$-th components.
For non-empty subsets $S$, $T$ of $[n]$, we write $S\le T$ to mean the relation: (the maximum element of $S$) $\le $ (the minimum element of $T$).
Denote by $S+1$ the subset $\{ k+1 \ |\ k\in S \}$ of $[n+1]$.

\begin{lemma}\label{Lem:ProperInt2}
Keep the notation above and let $b\in \{ 1,\dots ,i \} $.
Assume that the Cartier divisors 
$\Gamma \supero{a}(v_{k_a}^{ [m]} )$  
($a=1,\dots ,i$) 
on $X_m^\circ $  
form a local complete intersection for every choice of indices
$0\le k_1\le \dots \le k_i\le m$,
and the same holds for divisors $\Gamma \supers{a}(v_{k_a}^{[m]} )$ on $X_m^*$.
Assume moreover that the $i$ Cartier divisors on $X_m\os $:
\[ 
\Gamma \supero{a}(v_{k_a}^{[m]}  ) \ (1\le a\le b-1),\quad  \Gamma ^{(b)\clubsuit }(v_{k_b}^{[m]}), \quad 
\Gamma \supers{a} (v_{k_a}^{[m]}) \ (b+1\le a\le i)
\]
form a local complete intersection for every choice of $0\le k_1\le \dots \le k_i\le m$ and $\clubsuit \in \{ \circ ,* \}$.
Then for every choice of non-empty subsets
$S_1\le \dots \le S_i$ of $[m]$
and $k\in [m]$ with $S_b\le \{ k\} \le S_{b+1}$, the Cartier divisors on $X_m\os \times \square ^{m+1}$:
\[ \begin{array}{rlll}
\mrm{pr}^*_{S_a}\Gamma \supero{a}( S_a  ) 
\quad 1\le a\le b-1; 
&\mrm{pr}^*_{S_b\cup \{ k+1\} }\Gamma \superos{b}( S_b  ) ; \\%
\mrm{pr}^*_{S_a +1}\Gamma \supers{a} ( S_a ) \quad 
b+1\le a\le i\phantom{;}
\end{array}\]
form a local complete intersection, and the intersection satisfies the face condition.
Consequently, their intersection product belongs to $ z^i(X_m\os ,m+1)$.
\end{lemma}

\subsection{Bloch's specialization map \cite[p.292]{Bl86}}\label{BlochSp}

Here we give a precise argument to define the specialization map $\mrm{sp}_{L/k}$ in \S \ref{SecSpecialization}.
Let $L=k(x)$ be an extension with transcendence degree $1$ equipped with a basis $x$.
Then the presheaf $p_*z^i_{\rel ,k(x)}$ on $\MSm ^*/B\GL _r $ is contained in the following presheaf $p_*Z^i_{\rel ,k(x)}$:
\[ 
	\mfr{X}=((X,D),\lambda ,f;n,\alpha )\mapsto 
	\begin{cases} 
		p_*z^i_\rel (\mfr{X}\otimes _k k(x)) &\text{ if }D=\emptyset \\{} {} \\{}
		\displaystyle \frac{p_*z^i_\rel (\mfr{X}\otimes _k k[x]_{(x)})}{p_*z^{i-1}_\rel (\mfr{X})} &\text{ if }D\neq \emptyset ,
	\end{cases}   
\]
where we embed $p_*z^{i-1}_\rel (\mfr{X})$ into $p_*z^i_{\rel }(\mfr{X}\otimes _k {k[x]_{(x)}})$ by $\{ x=0 \} $.
We have an obvious scalar extension map $\mrm{res}_{k(x)/k}$: $p_*z^i_{\rel }\to p_*Z^i_{\rel ,k(x)}$.
We also consider the presheaf $p_*z^i_{\rel ,k[x]_{(x)}}$ defined by $\mfr{X}\mapsto p_*z^i_\rel (\mfr{X}\otimes _k {k[x]_{(x)}})$.

Now consider the sequence
\[
	0\to p_*z^{i}_\rel \xrightarrow{\{ x=0\} }
	p_*z^{i+1}_{\rel ,k[x]_{(x)}} \to
	p_*Z^{i+1}_{\rel ,k(x)} \to 0.
\]
For $\mfr{X}$ with $D\neq \emptyset $, this sequence is degreewise exact for a tautological reason.
If $D=\emptyset $, it is acyclic as a double complex by the localization theorem \cite[Theorem 0.1]{Bl94} (only known to be true  when $D=\emptyset $!).

The cycle $\Gamma _x=\{ 1+t(x-1)=0 \} $ in $\Spec (k(x)[t])$ represents $x\in k(x)^*=\CH ^1(\Spec (k(x)),1)$.
Denote its closure in $\Spec (k[x]_{(x)}[t])$ by $\bar{\Gamma }_x$.
The map 
$\bar{\Gamma }_x\cupp (-)\colon  p_*z^i_{\rel ,k(x)}\to p_*Z^{i+1}_{\rel ,k(x)}[-1]$
defined by
\[ 	V \mapsto 
	\begin{cases}
		\Gamma _x\cupp V &\text{ if }D=\emptyset \\
		\bar{\Gamma }_x\cupp V_{k[x]_{(x)}} &\text{ if }D\neq \emptyset 
	\end{cases}
 \]
is a well-defined map of complexes.
We define the specialization map $\mrm{sp}_{k(x)/k}\colon p_*z^i_{\rel ,k(x)}\to p_*z^i_{\rel }$ 
in $D(\MSm ^*/B\GL _r )$
by the zig-zag:
\[ \begin{array}{rcl}
	{}&p_*z^i_{\rel ,k(x)} \xrightarrow{\Gamma _x\cupp (-)} &p_*Z^{i+1}_{\rel ,k(x)} [-1]\\{}
	{}&{}&\downarrow \\{}
	p_*z^i_{\rel }\xrightarrow[\{ x=0 \} ]{\sim }
	&\mrm{cone}\Bigl( p_*z^{i+1}_{\rel ,k[x]_{(x)}} \to &p_*Z^{i+1}_{\rel ,k(x)}\Bigr) [-1] .
\end{array} \]
Of course, its composition with $\mrm{res}_{k(x)/k}$ gives the identity map on $p_*z^i_\rel $.
We leave it to the reader to verify this.

\begin{remark}
The specialization map depends on the choice of the transcendental basis.
For example, the specialization map
\begin{equation*}
\CH ^1(\Spec (k(x)),1)=k(x)^* \longrightarrow 
\CH ^1(\Spec (k),1)=k^*
 \end{equation*}
with respect to the basis $ax$ ($a\in k^*$) maps
$1/x$ to $a$.

\end{remark}

For a purely transcendental finitely generated extension $k(x_1,\dots ,x_r)/k$ with a chosen basis,
we define the specialization map by composition
\begin{equation*}    
	\mrm{sp}_{k(x_1,\dots ,x_r)/k}:=\mrm{sp}_{k(x_1)/k}\circ \dots \circ \mrm{sp}_{k(x_1,\dots ,x_r)/k(x_1,\dots ,x_{r-1})} .
	\end{equation*}
%
Beware that this map depends on the order on the transcendental basis. For example, the map $\mrm{sp}_{k(x)/k}\circ \mrm{sp}_{k(x,y)/k(x)}$:
\begin{equation*}
\CH ^1(\Spec (k(x,y)),1)=k(x,y)^*
\longrightarrow 
\CH ^1(\Spec (k),1)=k^*
 \end{equation*}
maps $ax+by$ ($a,b\in k^*$)
to $a$, while $\mrm{sp}_{k(y)/k}\circ \mrm{sp}_{k(x,y)/k(y)}$
maps it to $b$.